\documentclass[a4paper, 11pt]{article}
\usepackage[english]{babel}
\usepackage[utf8]{inputenc}

\usepackage[scaled=0.92]{helvet} 
\usepackage{lmodern}
\usepackage[T1]{fontenc}
\DeclareUnicodeCharacter{2212}{-}
\usepackage{geometry}
\usepackage[hyphens]{url}   
\usepackage{hyperref}       
\usepackage{graphicx}
\usepackage[font={it}]{caption}
 
\usepackage{wrapfig}

\usepackage{enumerate}
\usepackage{fancyhdr}
\usepackage{microtype}

\usepackage{extpfeil}

\usepackage{amsmath}
\allowdisplaybreaks
\usepackage{amssymb}
\usepackage{graphicx}
\usepackage{amsfonts}
\usepackage{amsthm}
\usepackage{pgfplots}
\pgfplotsset{compat=newest}
\usepackage{afterpage}
\usepackage{xfrac}
\usepackage{subcaption}
\usepackage{natbib}
\usepackage{lipsum}
\usepackage{enumitem}
\usepackage{booktabs}
\usepackage{varwidth}

\pgfmathdeclarefunction{rosenbrock}{2}{%
	\pgfmathparse{(1-#1)^2 + 100*(#2-#1^2)^2}%
}

\usepackage{titlesec}

\makeatletter
\renewcommand\maketitle{
  \begin{center}
    {\large\scshape\@title\par} 
    \vskip 1em 
    {\small\normalfont\scshape\@author\par} 
    \vskip 1em 
    {\normalfont\@date\par} 
  \end{center}
}
\makeatother

\usepackage[affil-it]{authblk} 
\usepackage{etoolbox}
\usepackage{lmodern}

\numberwithin{equation}{section}

\newtheorem{definition}{Definition}[section]
\newtheorem{theorem}[definition]{Theorem}
\newtheorem{assumption}[definition]{Assumption}
\newtheorem{condition}[definition]{Condition}

\newtheorem{lemma}[definition]{Lemma}

\usepackage{aligned-overset}
\usepackage{array}
\usepackage{cleveref}

\usepackage[boxed, ruled, linesnumbered]{algorithm2e}

\usepackage{tabularx}
\usepackage{afterpage}

\usepackage{authblk}

\usepackage[nolist]{acronym}

\usepackage[hang,flushmargin]{footmisc}
\usepackage{ragged2e}   
\usepackage{etoolbox}   

\appto\footnotesize{\RaggedRight}

\date{}
\begin{document}

\begin{acronym}
\acro{agc}[AGC]{approximated generalized Cauchy}
\acro{av}[avg.]{average}
\acro{bfgs}[BFGS]{Broyden–Fletcher–Goldfarb–Shanno}
\acro{foc}[FOC]{first-order critical}
\acro{fom}[FOM]{full order model}
\acro{hktr}[HKTR]{Hermite kernel trust-region}
\acro{imq}[IMQ]{inverse Multitquadric}
\acro{mse}[MSE]{mean-squared error}
\acro{pde}[PDE]{partial differential equation}
\acro{pd}[p.d.]{positive definite}
\acro{phktr}[PHKTR]{projected Hermite kernel trust-region}
\acro{rkhs}[RKHS]{Reproducing Kernel Hilbert Space}
\acro{spd}[s.p.d.]{strictly positive definite}
\acro{tr}[TR]{trust-region}
\end{acronym}

\pagenumbering{gobble}
\addtocontents{toc}{\protect\enlargethispage{\baselineskip}}
\pagenumbering{arabic}
\selectlanguage{\english}

\title{A trust-region framework for optimization using \\
Hermite kernel surrogate models}
\author{Sven Ullmann, Tobias Ehring, Robin Herkert, Bernard Haasdonk}
\affil{Institute of Applied Analysis and Numerical Simulation, University of Stuttgart, Pfaffenwaldring 57, 70569
Stuttgart, Germany}

{\renewcommand{\thefootnote}{}
\footnote{\footnotesize Email addresses: \texttt{sven.ullmann@mathematik.uni-stuttgart.de}, \texttt{tobias.ehring@mathematik.uni-stuttgart.de}, \texttt{robin.herkert@mathematik.uni-stuttgart.de}, \texttt{haasdonk@mathematik.uni-stuttgart.de}}%
\addtocounter{footnote}{-1}}

  \date{\today}
  \maketitle
  \begin{abstract}
In this work, we present a trust-region optimization framework that employs Hermite kernel surrogate models. The method targets optimization problems with computationally demanding objective functions, for which direct optimization is often impractical due to expensive function evaluations. To address these challenges, we leverage a trust-region strategy, where the objective function is approximated by an efficient surrogate model within a local neighborhood of the current iterate. In particular, we construct the surrogate using Hermite kernel interpolation and define the trust-region based on bounds for the interpolation error. As mesh-free techniques, kernel-based 
methods are naturally suited for medium- to high-dimensional problems. Furthermore, the Hermite formulation incorporates gradient information, enabling precise gradient estimates that are crucial for many optimization algorithms. We prove that the proposed algorithm converges to a stationary point, and we demonstrate its effectiveness through numerical experiments, which illustrate the convergence behavior as well as the efficiency gains compared to direct optimization.
  \end{abstract} 

\vspace{10pt}


  
\section{Introduction}
\label{chapter_introduction}

Optimization methods are essential tools across a wide variety of scientific 
domains. Examples include optimal (material) design in engineering, finding optimal molecular configurations in physics and chemistry or profit maximization in economics. In each of these fields, the goal is to determine a parameter 
in an admissible set that minimizes a given objective function. An 
(unconstrained) optimization problem can typically be stated as
\begin{align}
\label{eqn:optimization_problem}
\min_{\mu \in \mathcal{P}} J(\mu),
\end{align}
where \(J: \mathcal{P} \rightarrow \mathbb{R}\) is a real-valued objective function 
defined on the parameter set \(\mathcal{P} \subseteq \mathbb{R}^p\). A prototypical instance arises in the context of \ac{pde}-constrained optimization:
\begin{align}
\label{pde-constraint minimization problem}
\min_{\mu \in \mathcal{P}} \mathcal{J}(u(\cdot;\mu);\mu),
\end{align}
where $\mathcal{J}: H \times \mathcal{P} \rightarrow \mathbb{R}$ and $u(\cdot;\mu) \in H$ satisfies a \ac{pde}-constraint. The \ac{pde} is typically given in variational form
\begin{align}
\label{primal_eqn}
a(u(\cdot;\mu),v;\mu) \;=\; f(v;\mu)
\qquad\forall \; v\in V,
\end{align}
where $a(\cdot, \cdot; \mu): H \times V \rightarrow \mathbb{R}$ is a parameter-dependent bilinear or nonlinear form, and $f(\cdot;\mu): V \rightarrow \mathbb{R}$ is a  parameter-dependent linear form, both defined over appropriate function spaces $H$ and $V$. This formulation subsumes a wide class of linear and nonlinear \ac{pde}s, including elliptic, parabolic and hyperbolic problems. If there exists a unique solution \(u(\cdot;\mu) \in H\) of  \eqref{primal_eqn} for every parameter $\mu \in \mathcal{P}$, then the optimization problem 
\eqref{pde-constraint minimization problem} can be reformulated in form \eqref{eqn:optimization_problem} using \(J(\mu) := \mathcal{J}(u(\cdot;\mu);\mu).\) In every iteration of the optimization algorithm, the \ac{pde} defined in (\ref{primal_eqn}) must be solved. In real-world applications, however, the computational expenses of frequently solving the \ac{pde} renders a straightforward optimization scheme often impractical. \\

One approach to make these methods more computationally feasible is the so-called \ac{tr} approach. It constrains the search for the subsequent iterate to a localized neighborhood around the current iterate, known as the \ac{tr}. Within this region, the objective function is replaced by a surrogate model, which is designed to be more efficient to evaluate than the original objective function. \ac{tr} methods have been successfully applied in a wide range of scientific fields, including engineering \cite{Kwok1985LocationOS}, physics \cite{Barakat_92}, chemistry \cite{Jensen1994} and also in economics \cite{Studer_portfolio}. Moreover, these methods have been extended to scenarios involving inexact evaluations, such as approximate solutions of linear systems or gradient approximations within sequential least squares frameworks, as demonstrated by \cite{Heinkenschloss2002}. \\

The surrogate model which is used to approximate the objective function within the \ac{tr} plays a crucial role for the effectiveness of the algorithm. Numerous approaches for constructing such models exist, and a general framework for smooth models is presented in \cite{Alexandrov1998}. A comprehensive analysis  using quadratic surrogate models is provided in \cite[Chapter 6]{trust_region_methods}. In the context of \ac{pde}-constrained optimization, model order reduction techniques emerged as an effective way to construct the surrogate model, see e.g., \cite{Keil_2021, Qian2017ACT,  WenZahr2025, YueMeerbergen}. Further, \cite{Kartmann2024} proposes an \ac{tr} framework tailored to iterative regularization methods for inverse problems governed by elliptic \ac{pde}s. \\ 

Kernel methods \cite{wendland_2004} are powerful tools in surrogate modeling, that perform well in many applications, compare \cite{carlberg2019recovering, Denzel2019, doeppel2024goal,scholkopf2002learning}. These methods perform especially well for medium- to high-dimensional problems, as they are meshless, making them less susceptible to the curse of dimensionality. Additionally, they are used extensively in machine learning applications, e.g., in Support Vector Machines for classification, as discussed in \cite{Steinwart2008SupportVM}. Function approximation with standard kernel methods can be improved by using Hermite interpolation \cite[Chapter 16]{wendland_2004}, which interpolates function values as well as the gradients.  \\

In this work, therefore, we leverage Hermite kernel methods to construct the surrogate model used to approximate the objective function within the \ac{tr}. Our key contributions are: \begin{enumerate}
    \item We introduce the \ac{hktr} algorithm (Algorithm \ref{kernel_TR_algorithm}), 
    \item we construct the \ac{tr} not by using balls, as is common in the literature, but by employing the upper bound of the Hermite kernel interpolation error,
    \item we prove convergence of the \ac{hktr} algorithm in 
    \Cref{section_convergence_theory_kernel_TR}.
    
\end{enumerate}

This work is structured as follows: In \Cref{sect:Hermite}, we provide an essential background on kernel functions and depict elementary results mainly regarding Hermite kernel interpolation. In \Cref{section_TR}, we first present general \ac{tr} algorithms. Then, we introduce the \ac{hktr} algorithm and provide a convergence statement, which are the key contributions of this work. \Cref{chapter_numerical_examples} contains numerical examples of specific instances of (\ac{pde}-constrained) optimization problems to illustrate the functionality of the \ac{hktr} algorithm. Our work is concluded in \Cref{sect:Conclusion}. 

\section{Introduction to Hermite kernel interpolation}
\label{sect:Hermite}
We begin by reviewing some fundamental insights about kernel methods. For additional details, see \cite{wendland_2004}. A symmetric function \( k : \Omega \times \Omega \rightarrow \mathbb{R} \), defined on a non-empty set \( \Omega \subseteq \mathbb{R}^N \), is referred to as a kernel. 
A kernel is called \ac{pd} if, for every finite pairwise distinct set \( X_n := \{x_1, \dots, x_n\} \subset \Omega \), the Gram matrix \( \mathcal{K}_{X_n} := \left(k(x_i, x_j) \right)_{i,j=1}^n  \in \mathbb{R}^{n \times n} \) is positive semidefinite. Furthermore, if all such Gram matrices are \ac{pd}, the kernel is referred to as \ac{spd}. Clearly, all \ac{spd} kernels are also \ac{pd} kernels. Kernels that are \ac{pd} are of particular interest as they are uniquely associated with a \ac{rkhs}, denoted by \( \mathcal{H}_k(\Omega) \). An \ac{rkhs} is a Hilbert space of functions \( f : \Omega \rightarrow \mathbb{R} \) with the property that there exists a function \( k : \Omega \times \Omega \rightarrow \mathbb{R} \) such that \( k(x, \,\cdot\,) \in \mathcal{H}_k(\Omega) \) for all \(\; x \in \Omega \) and 
\begin{align}\label{eq:repProp}
\left\langle f, k(x, \,\cdot\,) \right\rangle_{\mathcal{H}_k(\Omega)} = f(x) \textnormal{ for all }  f \in \mathcal{H}_k(\Omega). 
\end{align} 
This is known as the reproducing property and $k$ is the reproducing kernel. Moreover, if $k \in C^2(\Omega \times \Omega)$, then for all $f \in \mathcal{H}_k(\Omega)$, it holds
\begin{align}\label{eq:repPropDeriv}
   \partial^{l}f(x) = \left\langle \partial^{l}_1 k(x,\cdot) ,f \right\rangle_{\mathcal{H}_k(\Omega)} \textnormal{ for all } f \in \mathcal{H}_k(\Omega) \;\;\;\, \forall\; l=1,...,N,
\end{align}
where $\partial^{l}_1$ denotes the partial derivative operator in direction $l$ w.r.t. its first argument. This result is a consequence of the reproducing property and the differentiability of the kernel (see \citep[Theorem 10.45]{wendland_2004}). In the following, to accommodate directional derivatives and the case where no differentiation is applied, we adopt the multi-index notation $ a \in \mathbb{N}^N_0 \textnormal{ with } \Vert a \Vert_1 \leq 1$ in the operator $\partial^{a}_1$.  Note that throughout the work $\Vert \cdot \Vert := \Vert \cdot \Vert_2$ will be denoted as the Euclidean norm. \\
\newline
We proceed with the formulation of Hermite kernel interpolation, a specific instance of generalized kernel interpolation as described in \cite[Chapter 16]{wendland_2004}. Hermite interpolation assumes access to both the values of a target function $f \, : \, \Omega \rightarrow \mathbb{R}$ and  its gradient $ \nabla f \, : \, \Omega \rightarrow \mathbb{R}^N$. For an \ac{spd} kernel $k$ with $k \in C^2(\Omega \times \Omega )$ and a finite pairwise distinct set $X_n = \{x_1,...,x_n \} \subset \Omega$, the objective of Hermite kernel interpolation is to construct a surrogate function $s_f^n$ that satisfies the following constrained minimization problem, 
\begin{align}
\label{eq:interDeriv}
\min_{s_f^n \in\mathcal{H}_k(\Omega) } \left\{ \Vert s_f^n \Vert_{\mathcal{H}_k(\Omega)} \; | \, \partial^{a} s_f^n(x) = \partial^{a} f(x);\; x \in X_n ;\; a \in \mathbb{N}^N_0 \textnormal{ with } \Vert a \Vert_1 \leq 1 \right\},
\end{align}
 where the conditions enforce interpolation of both the function values and their derivatives up to first order. The solution to this infinite-dimensional optimization problem is referred to as the minimal norm interpolant.
The solution admits a finite-dimensional representation, expressed as
\begin{align*}
s_f^n(x) =  \sum_{i=1}^n \alpha_i k(x_i,x) + \left\langle \beta_i , \nabla_1 k(x_i,x) \right\rangle_2,
\end{align*}
where $\langle \cdot, \cdot \rangle_2$ denotes the Euclidean inner product. The coefficients $\{\alpha_i\}_{i=1}^n \subset \mathbb{R}$ and $\{\beta_i\}_{i=1}^n \subset \mathbb{R}^N$ are determined by solving the system of linear equations
\begin{align}\label{eq:matrixM}
\mathcal{M}_{X_n} \begin{bmatrix}
\underline{\alpha} \\ \underline{\beta} 
\end{bmatrix} = \begin{bmatrix}
\underline{s_f^n(X_n)} \\ \, \underline{\nabla s_f^n(X_n)}
\,\end{bmatrix}.
\end{align}
The latter represents the interpolation condition in \eqref{eq:interDeriv}. For more details and a precise definition of the generalized Gram matrix $\mathcal{M}_{X_n}$ we refer to \cite{hermite_kernel_interpolation}. In particular, if the kernel $k$ is assumed to be an \ac{spd} translationally invariant kernel, i.e., $$k(x,y) = \phi(x-y)\, \textnormal{ for }\, x,y \in \Omega,$$ with $\phi\in C^{2}(\Omega)\cap L^1(\Omega)$, then the matrix $\mathcal{M}_{X_n}$  is symmetric positive definite for all pairwise distinct $X_n \subset \Omega$ (see \cite[Proposition 1]{hermite_kernel_interpolation}) and therefore the coefficients  $\{\alpha_i\}_{i=1}^n \subset \mathbb{R}$ and $\{\beta_i\}_{i=1}^{n} \subset \mathbb{R}^N$ are uniquely determined.
In this case, the Hermite interpolant can also be obtained via the orthogonal projection $$\Pi_{V(X_n)}: \mathcal{H}_k(\Omega) \rightarrow V(X_n)$$ of the RKHS $\mathcal{H}_k(\Omega)$ onto the closed subspace $$ V(X_n):=  \textnormal{span} \left\{ \partial^a_1 k(x,\cdot) \; | \; x \in X_n ;\, a \in \mathbb{N}^N_0 \textnormal{ with } \Vert a \Vert_1 \leq 1 \right\} \subset \mathcal{H}_k(\Omega).$$  This results directly from the fact that $\Pi_{V(X_n)}f$ is an interpolant, as for any  $x \in X_n$, it holds
\begin{align*}
 \partial^a \left( \Pi_{V(X_n)}f(x) \right) &= \left\langle   \partial^a_1 k(x,\cdot),\Pi_{V(X_n)}f \right\rangle_{\mathcal{H}_k(\Omega)}\\ &= -\underbrace{\left\langle   \partial^a_1 k(x,\cdot),\left(I-\Pi_{V(X_n)}\right)f \right\rangle_{\mathcal{H}_k(\Omega)}}_{=0 \;(\textnormal{orthogonality of projection error)}} + \left\langle   \partial^a_1 k(x,\cdot),f \right\rangle_{\mathcal{H}_k(\Omega)} = \partial^a_1 f(x),
\end{align*}
where property \eqref{eq:repPropDeriv},  $\partial^a_1 k(x,\cdot) \in V(X_n)$ for $x \in X_n$ and \eqref{eq:repProp} were utilized. Additionally,  $\Pi_{V(X_n)}f$ minimizes the norm  among all interpolants $s \in \mathcal{H}_k(\Omega)$, since with Pythagoras, we have
\begin{align*}
    \left\Vert s \right\Vert_{\mathcal{H}_k(\Omega)}^2 &= \left\Vert s -\Pi_{V(X_n)}s + \Pi_{V(X_n)}s \right\Vert_{\mathcal{H}_k(\Omega)}^2 \\ 
    &= \left\Vert s -\Pi_{V(X_n)}s \right\Vert_{\mathcal{H}_k(\Omega)}^2 + \left\Vert\Pi_{V(X_n)}s \right\Vert_{\mathcal{H}_k(\Omega)}^2 \geq \left\Vert\Pi_{V(X_n)}s \right\Vert_{\mathcal{H}_k(\Omega)}^2 = \left\Vert\Pi_{V(X_n)}f \right\Vert_{\mathcal{H}_k(\Omega)}^2,
\end{align*}
where the last equality follows from 
\begin{align*}
\left\Vert\Pi_{V(X_n)}f-\Pi_{V(X_n)}s \right\Vert_{\mathcal{H}_k(\Omega)}^2 &= 
    \left\langle \Pi_{V(X_n)}\left(f-s\right),\Pi_{V(X_n)}\left(f-s\right) \right\rangle_{\mathcal{H}_k(\Omega)}^2 \\ &=     \left\langle \Pi_{V(X_n)}\left(f-s\right),f-s \right \rangle_{\mathcal{H}_k(\Omega)}^2 \\ 
    &= \left\langle \sum_{i=1}^n \tilde{\alpha}_i k(x_i,x) 
 +\left\langle \tilde{\beta}_i, \nabla_1 k(x_i,x) \right\rangle_2, f-s  \right\rangle_{\mathcal{H}_k(\Omega)}^2 \\
    &= \sum_{i=1}^n \tilde{\alpha}_i \underbrace{\left(f(x_i)-s(x_i)\right)}_{=0} + \langle \tilde{\beta}_i , \underbrace{\nabla  f(x_i)- \nabla s(x_i) }_{=0} \rangle_2 = 0
\end{align*}
for some appropriate  coefficients $\{\tilde{\alpha}_i\}_{i=1}^n \subset \mathbb{R}$ and $\{\tilde{\beta}_i\}_{i=1}^{n} \subset \mathbb{R}^N$. The last equality follows from the reproducing properties \eqref{eq:repProp} and \eqref{eq:repPropDeriv}.  \\

A crucial aspect in the application of Hermite kernel surrogates within the later introduced \ac{hktr} algorithm is the quantification of the point-wise interpolation error. In (Hermite) kernel interpolation, the primary tool for this purpose is the (Hermite) Power function:

\begin{definition}((Hermite) Power function) \label{def:hermite_power_function} \\
Let $\Omega \subseteq \mathbb{R}^N$ be non-empty,  $k \in C^{2}(\Omega \times \Omega)$ an \ac{spd} kernel and  $X_n = \{ x_j \}_{j=1}^n \subset \Omega$ be a pairwise distinct point set, then for a multi-index $a \in \mathbb{N}^N_0$ with $\Vert a \Vert_1 \leq 1$  the Hermite Power function $P^a_{X_n}: \Omega \rightarrow \mathbb{R}$ is given by \begin{align*}
P^a_{X_n}(x) := \left\Vert \left(I-\Pi_{V(X_n)}\right) \left(\partial_1^ak(x,\cdot) \right) \right\Vert_{\mathcal{H}_k(\Omega)}.
\end{align*}
We further define $P_{X_n} := P_{X_n}^0$. 
\end{definition}
With this definition of the Hermite Power function, we obtain the following point-wise error bound on the interpolation error
\begin{align} 
    \left|\partial_1^af(x) - \partial_1^a\left(\Pi_{V(X_n)} f \right)(x) \right|
    & = \left|\left\langle \partial_1^ak(x,\cdot)  , (I- \Pi_{V(X_n)} ) f  \right\rangle_{\mathcal{H}_k(\Omega)} \right| \nonumber \\
    &= \left|\left\langle f , \left(I-\Pi_{V(X_n)}\right)\left(\partial_1^ak(x,\cdot) \right) \right\rangle_{\mathcal{H}_k(\Omega)} \right| \nonumber \\ 
    &\leq 
    \Vert f \Vert_{\mathcal{H}_k(\Omega)} P^a_{X_n}(x)  \label{eqn:def:interpolation error}
\end{align}
for all $x \in \Omega$. Moreover, the gradient error can be bounded by aggregating over all directional derivatives of order $\Vert a \Vert_1 = 1$. Specifically:
\begin{align*}
\left\Vert\nabla f(x) - \nabla \left(\Pi_{V(X_n)} f\right)(x) \right\Vert &\leq \sqrt{\sum_{ a \in \mathbb{N}^N_0,\, \Vert a \Vert_1 = 1}   \left(P^a_{X_n}(x)\right)^2 \Vert f \Vert^2_{\mathcal{H}_k(\Omega)}  }.
\end{align*}

Note that for a function $f \in \mathcal{H}_k(\Omega) $ it is generally not possible to compute $\Vert f\Vert_{\mathcal{H}_k(\Omega)}$. However, the \ac{rkhs}-norm of the Hermite kernel interpolant $s_f^n$ can be computed via \begin{align}
\label{eq:computation_rkhs}
 \Vert s_f^n \Vert_{\mathcal{H}_k(\Omega)}^2 &= \left\langle \sum_{i=1}^n \alpha_i k(x_i,\cdot) + \left\langle \beta_i , \nabla_1 k(x_i,\cdot) \right\rangle_2,  \sum_{i=1}^n \alpha_i k(x_i,\cdot) + \left\langle \beta_i , \nabla_1 k(x_i,\cdot) \right\rangle_2 \right\rangle_{\mathcal{H}_k(\Omega)} \nonumber \\
&= \begin{bmatrix}
    \underline{\alpha} & \underline{\beta}
\end{bmatrix} \mathcal{M}_{X_n}\begin{bmatrix}
    \underline{\alpha} \\ \underline{\beta}
\end{bmatrix},
\end{align}
where $\mathcal{M}_{X_n}, \underline{\alpha}$ and $\underline{\beta}$ are defined in \eqref{eq:matrixM}. By increasing the amount of interpolation points, s.t. the fill-distance  \begin{align*}
h_{X_n} := \sup_{x \in \Omega} \min_{1 \leq i \leq n} \left\Vert x - x_i \right\Vert
\end{align*} converges to zero, it is possible to prove 
\begin{align*}
  \lim_{n \rightarrow \infty}   \Vert s_f^n-f \Vert_{\mathcal{H}_k(\Omega)} = 0  .
\end{align*}
Therefore, for every \( \epsilon > 0 \) there exists an \( N \in \mathbb{N} \) such that for all \( n > N \) the following inequality holds:
\[
  \| s_f^n \|_{\mathcal{H}_k(\Omega)} \leq \| f \|_{\mathcal{H}_k(\Omega)} \leq \| s_f^n - f \|_{\mathcal{H}_k(\Omega)} + \| s_f^n \|_{\mathcal{H}_k(\Omega)} \leq \epsilon + \| s_f^n \|_{\mathcal{H}_k(\Omega)}.
\]
This estimate implies that \begin{align}
\label{eq:estimate_rkhsnorm}
     \| s_f^n \|_{\mathcal{H}_k(\Omega)} \approx \| f \|_{\mathcal{H}_k(\Omega)},
\end{align}
an approximation that will be utilized in the numerical experiments. \\



The Power function  $P_{X_n}$  plays a central role in defining the \ac{tr} constraint in the proposed framework. Its properties are therefore critical to the convergence analysis of the  \ac{hktr} algorithm. In particular, it is essential that the Power function exhibits Hölder continuity. The following theorem establishes mild conditions under which this property holds.
\begin{theorem}(Hölder continuity of the Power function) \label{theorem_P_hölder}\\
Let $\Omega \subseteq \mathbb{R}^N$ be non-empty, $X_n = \{ x_j \}_{j=1}^n \subset \Omega$ be a pairwise distinct point set,  $k \in C^{2}(\Omega \times \Omega)$ be an \ac{spd} kernel with $k(x,\cdot): \Omega \rightarrow \mathbb{R}$ being uniformly Lipschitz continuous for all $x \in \Omega$, i.e., there exists a $C_k < \infty$ such that \begin{align*}
|k(x, \tilde{x}) - k(x, x')| \leq C_k \Vert\tilde{x} - x'\Vert \quad \forall \; \tilde{x}, x' \in \Omega,
\end{align*} then the Power function $P_{X_n}$ is Hölder continuous with $\alpha_{\textnormal{Höl}} = \sfrac{1}{2}$, i.e.,
\begin{align*}
\left\vert P_{X_n}(x) - P_{X_n}(y) \right\vert \leq 4 \sqrt{C_k} \Vert x-y \Vert^{\frac{1}{2}} \quad \forall \; x,y \in \Omega.
\end{align*}
\end{theorem}
\begin{proof}
    It holds
    \begin{align}
        \left| P_{X_n}(x) - P_{X_n}(y) \right| &= \left| \left\Vert k(\cdot, x) - \Pi_{V(X_n)}(k(\cdot, x)) \right\Vert_{\mathcal{H}_k(\Omega)} - \left\Vert k(\cdot, y) - \Pi_{V(X_n)}(k(\cdot, y)) \right\Vert_{\mathcal{H}_k(\Omega)} \right| \nonumber \\
        &\leq  \left| \left\Vert k(\cdot, x) - k(\cdot, y)  + \Pi_{V(X_n)}(k(\cdot, y)) - \Pi_{V(X_n)}(k(\cdot, x))\right\Vert_{\mathcal{H}_k(\Omega)} \right| \label{eqn:theorem_P_hölder1} \\
                &\leq   \left\Vert k(\cdot, x) - k(\cdot, y)\right\Vert_{\mathcal{H}_k(\Omega)}  + \left\Vert \Pi_{V(X_n)}(k(\cdot, y) -k(\cdot, x))\right\Vert_{\mathcal{H}_k(\Omega)} \label{eqn:theorem_P_hölder2} \\
                &  \leq  2 \Vert k(\cdot, x) - k(\cdot, y)\Vert_{\mathcal{H}_k(\Omega)} \label{eqn:theorem_P_hölder3}\\
                &  = 2 \sqrt{ k(x, x) - k(x, y) + k(y, y) - k(y, x)} \label{eqn:theorem_P_hölder4}\\
                &\leq 2 \sqrt{|k(x,x) - k(x,y)| + |k(y,y) - k(y,x)|} \label{eqn:theorem_P_hölder5} \\
                 &  \leq  2 \sqrt{ |k(x, x) - k(x, y)|} + 2\sqrt{|k(y, y) - k(y, x)|} \label{eqn:theorem_P_hölder6}\\
                 &  \leq  4 \sqrt{C_k} \Vert x-y \Vert^{\frac{1}{2}}, \label{eqn:theorem_P_hölder7}
    \end{align}
    where we used the inverse triangle inequality in \eqref{eqn:theorem_P_hölder1}, the triangle inequality in \eqref{eqn:theorem_P_hölder2}, the submultiplicativity of the norm together with the fact that $\Vert\Pi_{V(X_n)}\Vert_{\mathcal{L}\left(\mathcal{H}_k(\Omega),\mathcal{H}_k(\Omega)\right)} = 1$ in \eqref{eqn:theorem_P_hölder3}, the reproducing property of the \ac{rkhs} in \eqref{eqn:theorem_P_hölder4}, the monotonicity of the square root in \eqref{eqn:theorem_P_hölder5}, the fact that $\sqrt{a +b} \leq \sqrt{a} + \sqrt{b}$ for $a, b \geq 0$ in \eqref{eqn:theorem_P_hölder6}  and the uniform Lipschitz continuity of the kernel in \eqref{eqn:theorem_P_hölder7}. 
\end{proof}
The Lipschitz continuity of the gradient of the kernel surrogate model is another key property in the convergence analysis of the \ac{hktr} algorithm introduced later. This property can be ensured under relatively mild conditions on the kernel, as demonstrated by the following theorem.
\begin{theorem}(Lipschitz continuity of the Hermite kernel interpolant) 
\label{thm:LipschitzDerivKernel} \\
Let $\Omega \subseteq \mathbb{R}^N$ be non-empty and $k \in C^{2}(\Omega \times \Omega)$ be an \ac{spd} kernel  with $\partial_1^l\partial_2^l k(x,\cdot): \Omega \rightarrow \mathbb{R}$ being uniformly Lipschitz continuous for all $x \in \Omega$ and all $l=1,...,N$ with maximum Lipschitz constant $C_{\nabla k} \geq 0$, then  the gradient of the Hermite kernel interpolant $ s_f^n = \Pi_{V(X_n)} f$ for $f \in \mathcal{H}_{k(\Omega)}$ is uniformly Lipschitz continuous w.r.t. the set of interpolation points. Specifically, there exists a constant $C_{k, \nabla k, f, N}
\geq 0$ such that \begin{align*}
    \left\Vert \nabla \left(\Pi_{V(X_n)} f\right)(x) - \nabla \left(\Pi_{V(X_n)} f\right)(x') \right\Vert \leq C_{k, \nabla k, f, N} \Vert x - x' \Vert \quad \textnormal{ for all } x, x' \in \Omega
\end{align*} 
for all finite, pairwise distinct subsets $X \subset \Omega$, where  $C_{k, \nabla k, f, N} := 2 C_{\nabla k}\sqrt{N}\Vert  f \Vert_{\mathcal{H}_k(\Omega)}$. 
\end{theorem}

\begin{proof}
We have
\begin{align} 
    &\left\Vert\nabla \left(\Pi_{V(X_n)} f\right)(x) - \nabla \left(\Pi_{V(X_n)} f\right)(x') \right\Vert^2 \nonumber \\
    =& \sum_{l=1}^N  \left| \partial^l \left(\Pi_{V(X_n)} f\right)(x) - \partial^l \left(\Pi_{V(X_n)} f\right)(x') \right|^2 \nonumber \\
    =& \sum_{l=1}^N  \left \langle \partial^l_1 k(x,\cdot)- \partial^l_1 k(x',\cdot), \Pi_{V(X_n)} f \right \rangle_{\mathcal{H}_k(\Omega)}^2 \label{eqn:theorem_s_Lip1} \\
        \leq & \sum_{l=1}^N \left\Vert \partial^l_1 k(x,\cdot)- \partial^l_1 k(x',\cdot) \right\Vert_{\mathcal{H}_k(\Omega)}^2  \left\Vert  \Pi_{V(X_n)} f \right\Vert_{\mathcal{H}_k(\Omega)}^2  \label{eqn:theorem_s_Lip2}\\
                \leq & \Vert  f \Vert_{\mathcal{H}_k(\Omega)}^2 \sum_{l=1}^N \left\Vert \partial^l_1 k(x,\cdot)- \partial^l_1 k(x',\cdot) \right\Vert_{\mathcal{H}_k(\Omega)}^2 \label{eqn:theorem_s_Lip3} \\ 
                \leq & \Vert  f \Vert_{\mathcal{H}_k(\Omega)}^2 \sum_{l=1}^N  \left(\partial_1^l\partial_2^l k(x,x)- \partial_1^l\partial_2^l k(x,x') +\partial_1^l\partial_2^l k(x',x')- \partial_1^l\partial_2^l k(x',x) \right)^2 \label{eqn:theorem_s_Lip4}\\
          \leq& \Vert  f \Vert_{\mathcal{H}_k(\Omega)}^2\sum_{l=1}^N   \left( \left|\partial_1^l\partial_2^l k(x,x)- \partial_1^l\partial_2^l k(x,x') \right| +\left|\partial_1^l\partial_2^l k(x',x')- \partial_1^l\partial_2^l k(x',x)\right|\right)^2\label{eqn:theorem_s_Lip5}\\
          \leq& \Vert  f \Vert_{\mathcal{H}_k(\Omega)}^2\sum_{l=1}^N   \left(C_{\nabla k}\Vert x - x' \Vert +C_{ \nabla k} \Vert x - x' \Vert\right)^2 \label{eqn:theorem_s_Lip6} \\  
           \leq & \; 4C_{\nabla k}^2 \;  N \Vert  f \Vert_{\mathcal{H}_k(\Omega)}^2\Vert x - x' \Vert^2 \nonumber \\
           =& \; C_{k, \nabla k, f, N}^2 \Vert x - x' \Vert^2 \nonumber,
\end{align}
    where we used 
    \eqref{eq:repPropDeriv} in \eqref{eqn:theorem_s_Lip1}, the Cauchy-Schwartz inequality in \eqref{eqn:theorem_s_Lip2}, the submultiplicativity of the norm together with the fact that $\Vert\Pi_{V(X_n)}\Vert_{\mathcal{L}\left(\mathcal{H}_k(\Omega),\mathcal{H}_k(\Omega)\right)} = 1$ in \eqref{eqn:theorem_s_Lip3}, \eqref{eq:repPropDeriv} in \eqref{eqn:theorem_s_Lip4}, the  triangle inequality in \eqref{eqn:theorem_s_Lip5} and the uniform Lipschitz continuity of the kernel in \eqref{eqn:theorem_s_Lip6}.
\end{proof}

Popular kernels, that will be used in the numerical experiments in \Cref{chapter_numerical_examples} are: \begin{definition}(Widely used kernels)  \label{def:three_kernels} \\
The Gaussian kernel, defined as \begin{align*}
         k(x,x'; \varepsilon) &= \exp(-\varepsilon^2\Vert x - x'\Vert^2).
    \end{align*}
The quadratic Matérn kernel, defined as \begin{align*}
        k(x,x'; \varepsilon) &= (3 + 3\varepsilon  \Vert x - x'\Vert + \varepsilon^2\Vert x - x'\Vert^2)\exp(-\varepsilon \Vert x - x'\Vert).
    \end{align*}
The Wendland kernel of second order, defined as \begin{align*}
         k(x, x'; \varepsilon) &= \frac{(l + 4)!}{l!} \max(1 -\varepsilon\Vert x-x'\Vert, 0)^{l+2} \left( (l^2+ 4l +3)\varepsilon^2\Vert x-x'\Vert^2 + (3l + 6) \varepsilon\Vert x-x'\Vert + 3\right),
    \end{align*}
where  $l= \left\lfloor \sfrac{N}{2}\right\rfloor + 3$.

\end{definition} 

Note that these three kernels are \ac{spd} (Gaussian kernel: \cite[Theorem 6.10]{wendland_2004}, quadratic Matérn kernel: Can be concluded from \cite[Theorem 4.2] {Fasshauer_2011} together with \cite[Example 5.7]{Fasshauer_2011}, Wendland kernel of second order: \cite[Theorem 9.13]{wendland_2004}).
Furthermore, as so-called radial basis function kernels, they are also translation invariant. Additionally, these three kernels are at least in the function class $C^3( \Omega \times \Omega)$ and having the first, second, and third derivatives  bounded on $\mathbb{R}^N$. This guarantees with the mean value theorem that the uniformly Lipschitz continuity of $k(x,\cdot)$ and $\partial_1^l\partial_2^l k(x,\cdot)$ are satisfied. Therefore, the conditions of Theorem \ref{theorem_P_hölder} and Theorem \ref{thm:LipschitzDerivKernel} are satisfied for these kernels. 

\section{Hermite kernel trust-region algorithm}
\label{section_TR}

In \Cref{section_standard_TR}, we introduce the \ac{tr} method in a general context. For further details - particularly regarding quadratic 
surrogate models - we refer to \cite[Chapter 6]{trust_region_methods}. In \Cref{section_assumptions_TR}, we present the set of assumptions required for the convergence analysis of our proposed algorithm. Subsequently, in 
\Cref{section_hermite_TR_kernel}, we describe the \ac{hktr} algorithm, including its underlying optimization subproblem and the definition of the \ac{agc} point, which is crucial both for the convergence analysis and the practical implementation of the algorithm. We then establish the convergence theory for the \ac{hktr} method in 
\Cref{section_convergence_theory_kernel_TR}. For simplicity, we specialize to the case \(\mathcal{P} := \mathbb{R}^p\) throughout this section and 
defer any discussion concerning subsets \(\mathcal{P} \subset \mathbb{R}^p\) to \Cref{section_constrained_optimization}.  \\

In \Cref{sect:Hermite}, we employed notation typically used in kernel methods. For the remainder of this work, however, we adopt the standard notation of optimization theory. Accordingly, we will refer to the dimension \(N\) as \(p\), the set \(\Omega\) as \(\mathcal{P}\), the function to be approximated \(f\) as the objective function \(J\), the kernel interpolant 
\(s_f^n\) as \(\hat{J}^{(i)}\), and the interpolation point set $X_n$ consisting of centers \( \{x_i \}_{i=1}^n \subset \Omega\) as $M^{(i)}$ consisting of the iterates of the optimization method $\left \{\mu^{(j)} \right\}_{j=0}^i  \subset \mathcal{P}$.

\subsection{General trust-region algorithm}
\label{section_standard_TR}

We consider a procedure for solving the optimization problem \eqref{eqn:optimization_problem} by means of a \ac{tr} algorithm. The key idea behind this approach is to approximate the objective function \(J\) in a neighborhood of the current iterate \(\mu^{(i)}\) - known as the \ac{tr} - with a surrogate model \(\hat{J}^{(i)}\). This surrogate 
is intended to allow more efficient function evaluations than the original objective function $J$. Specifically, at iteration \(i\), the \ac{tr} is defined as \begin{align} \label{eqn:tr_sphere}
B^{(i)} := \left\{ \mu \in \mathcal{P} \; \middle| \; \left\Vert\mu - \mu^{(i)} \right\Vert \leq \delta^{(i)} \right\}, 
\end{align}
where $\delta^{(i)} > 0$ is the so-called \ac{tr} radius at the $i$-th iteration. We then compute a next potential iterate $\mu^{(i+1)} \in B^{(i)}$ that should decrease the surrogate model $\hat{J}^{(i)}$ sufficiently. 
Following this step, we assess whether the decrease 
predicted by the surrogate model is actually realized by \(J\). If this is the case, the iterate $\mu^{(i+1)}$ gets accepted - otherwise, it gets rejected and the  \ac{tr} radius $\delta^{(i)}$ shrinks, reflecting the assumption that \(\hat{J}^{(i)}\) may yield more accurate predictions in a smaller region. These steps are summarized in Algorithm \ref{standard_TR_algorithm}, which we refer to as a general \ac{tr} algorithm. Note that the algorithm is understood as an abstract procedure, as it has no finite termination condition. These are provided at the end of this subsection. 

\begin{algorithm}[h]
\caption{General trust-region (TR) algorithm}\label{standard_TR_algorithm}
\KwIn{Objective function $J$, initial iterate $\mu^{(0)}$, initial \ac{tr} radius $\delta^{(0)}$, constants $\xi_1$, $\xi_2$ \& $\beta$, s.t. $0 < \xi_1 < \xi_2 <1 $, $\beta \in (0,1)$.}
Set $i:=0$. \\
Construct a surrogate model $\hat{J}^{(i)}$ on $B^{(i)}$. \label{line_algo_standard_TR_gobackto} \\
Compute the next iterate $\mu^{(i+1)} \in B^{(i)}$, s.t. the surrogate model $\hat{J}^{(i)}$ will be sufficiently decreased at this next iterate. \\
Compute \begin{align}
\label{check_sufficient_decrease_of_model}
\rho^{(i)} := \frac{J(\mu^{(i)}) - J(\mu^{(i+1)})}{\hat{J}^{(i)}(\mu^{(i)}) - \hat{J}^{(i)}(\mu^{(i+1)})}
\end{align} \\
\uIf{ $\rho^{(i)} \geq \xi_1$}{Accept $\mu^{(i+1)}$ as the next iterate.}
\uElse{Reject and set $\mu^{(i+1)} := \mu^{(i)}$.} 
Update the \ac{tr} radius according to (\ref{updating_tr_radius}). \\
$i := i+1$ and go back to line \ref{line_algo_standard_TR_gobackto}. \\
\KwOut{Sequence of iterates $\{\mu^{(i)}\}_{i \in \mathbb{N}_0}$ }
\end{algorithm} 
\vspace{10pt}
Typical values for the constants in Algorithm \ref{standard_TR_algorithm} according to \cite{trust_region_methods} are: $\xi_1 = 0.1, \xi_2 = 0.9$ and $\beta = 0.5$. If the decrease in the surrogate model $\hat{J}^{(i)}$ and the decrease in the objective function $J$ almost coincide, i.e., $\rho^{(i)} \geq \xi_2$, we trust the surrogate model and therefore expand the \ac{tr} for the next iteration, s.t. $\delta^{(i+1)} := \beta^{-1} \delta^{(i)}$. We call these iterations \textit{very successful}. If we only obtain $\rho^{(i)} \geq \xi_1$, we still accept the new iterate $\mu^{(i)}$, but we neither shrink nor expand the \ac{tr} for the next iteration. These iterations are called \textit{successful}. In the last case, i.e., $\rho^{(i)} < \xi_1$, we do not accept the iterate $\mu^{(i+1)}$, as the decrease in the surrogate model $\hat{J}^{(i)}$ was not reflected in the objective function $J$. Therefore, we shrink the \ac{tr} for the next iteration, s.t. $\delta^{(i+1)} := \beta \delta^{(i)}$. We obtain the following update scheme for the \ac{tr} radius \begin{align}
\label{updating_tr_radius}
\delta^{(i+1)} = \begin{cases}
\beta^{-1} \delta^{(i)} \quad &\textnormal{ if } \rho^{(i)} \geq \xi_2 \\
\delta^{(i)} \quad &\textnormal{ if } \rho^{(i)} \in [\xi_1, \xi_2) \\
\beta \delta^{(i)} \quad &\textnormal{ if } \rho^{(i)} < \xi_1.
\end{cases}
\end{align}

The formulation of Algorithm \ref{standard_TR_algorithm} is kept very general. For example, it is not specified how to construct the surrogate model $\hat{J}^{(i)}$. As mentioned in the introduction, we will build $\hat{J}^{(i)}$ as a linear combination of kernel translates in the proposed \ac{hktr} algorithm. Another approach that is often used in practice is a quadratic model of the form \begin{align}
\label{quadratic_model_J}
\hat{J}^{(i)}(\mu) = J(\mu^{(i)}) + \left\langle g_i, \; d^{(i)}_\mu \right\rangle + \frac{1}{2} \left\langle d^{(i)}_\mu, \hat{H}_{\hat{J}^{(i)}}(\mu^{(i)}) \; d^{(i)}_\mu \right\rangle,
\end{align}
using the abbreviations $g_i := \nabla J(\mu^{(i)})$, $d^{(i)}_\mu := \mu - \mu^{(i)}$ and $\hat{H}_{\hat{J}^{(i)}}(\mu^{(i)})$ is a symmetric approximation of the Hessian $H_J(\mu^{(i)})$. 
It is also not specified how to compute the next iterate $\mu^{(i+1)}$ in the current \ac{tr} $B^{(i)}$. We comment on that in \Cref{section_hermite_TR_kernel} in the case of the \ac{hktr} algorithm. \\

So far, the algorithm has been presented without explicit termination 
criteria, which is impractical. In a computational setting, various termination criteria may be employed, 
for example: \begin{enumerate}
\item \textit{Maximum iterations:} A maximum amount of iterations $i_{max}$ or
\item \textit{FOC condition:}  $\left\Vert\nabla J(\mu^{(i)})\right\Vert_{\infty} \leq \tau_{\textnormal{FOC}}$ for some constant $\tau_{\textnormal{FOC}} \ll 1$, i.e., $\mu^{(i)}$ is close to a \ac{foc} point $\mu^*$, meaning $\mu^*$ satisfies $\left\Vert\nabla J(\mu^*)\right\Vert = 0$ or 
\item \textit{Stagnation:} $ J_{\textnormal{diff}} \leq \tau_J$ for some constant $\tau_J \ll 1$, where \begin{align*}
J_{\textnormal{diff}} := \frac{J(\mu^{(i)}) - J(\mu^{(i+1)})}{\max \left\{J(\mu^{(i)}), J(\mu^{(i+1)}), 1 \right\}}, 
\end{align*}
i.e., the algorithm terminates if no significant improvement was achieved by the new iterate $\mu^{(i+1)}$.
\end{enumerate} 

\subsection{Assumptions on the optimization problem}
\label{section_assumptions_TR}
The convergence analysis in \Cref{section_convergence_theory_kernel_TR} 
requires assumptions on both the objective function \(J\) being optimized and on the surrogate model \(\hat{J}^{(i)}\) constructed by Algorithm~\ref{kernel_TR_algorithm}. Consequently, the present section lists the necessary assumptions on \(J\) and \(\hat{J}^{(i)}\) and explains why 
they are both required and reasonable.
 


\begin{assumption}(Assumptions on $J$ and $\hat{J}^{(i)}$)
\label{assumption_list}
\begin{enumerate}[label=(\alph*)]
    \item The surrogate model $\hat{J}^{(i)}$ is twice differentiable for all iterations $i$. 
    \textnormal{We will use the \ac{bfgs} method to solve the subproblem (\ref{suboptimzation_problem_abstract}) presented in the next section. In order for a quasi-Newton scheme to converge, $C^2(\mathcal{P})$ of the function is required. This assumption is satisfied for the Hermite kernel surrogate model $\hat{J}^{(i)}$ using any of the kernel stated in \Cref{def:three_kernels}, as they are all at least in $C^3(\mathcal{P} \times \mathcal{P})$.}
    \item The objective function $J$ is uniformly bounded away from zero, i.e, there exists $c > 0$ s.t. $J(\mu) > c > 0$ for all parameters $\mu \in \mathcal{P}$. \textnormal{This assumption is not restrictive, as the boundedness from below is a usual assumption in minimization problems for physical applications, e.g., if $J(\mu)$ is an energy function. Therefore, if a lower global bound exists for $J(\mu)$, we can add a sufficiently large constant without changing the position of its local minima and ensure its strict positivity, cf.\,\cite{Keil_2021}.}
    \item For all iterations $i$, the kernel surrogate model $\hat{J}^{(i)}(\mu)$ is uniformly (w.r.t $\mu$ and $i$) bounded away from zero, i.e., there exists $c > 0$ s.t. $\hat{J}^{(i)}(\mu) > c > 0$ for all $\mu \in \mathcal{P}, i \in \mathbb{N}$.  \textnormal{Given that $\hat{J}^{(i)}$ is designed to approximate the objective function $J$ within the current \ac{tr}, this assumption appears justified. Moreover, we note that there exist techniques which, by construction, ensure this property for the (Hermite) kernel surrogate globally, cf.\,\cite[Section 3]{hermite_kernel_interpolation}.}
    \item  We require that $J \in \mathcal{H}_k(\mathcal{P})$. \textnormal{This assumption plays a crucial role, as it enables the estimation of the upper bound on the interpolation error stated in \eqref{eqn:def:interpolation error}, which is a fundamental component of Algorithm \ref{kernel_TR_algorithm}. Certainly not every kernel will be a suitable choice for every objective function $J$.}
    \item For all iterations $i$, the kernel surrogate model $\hat{J}^{(i)}$ as well as its gradient $\nabla \hat{J}^{(i)}$ are Lipschitz continuous on the parameter space $\mathcal{P}$. \textnormal{The reason behind assuming Lipschitz continuity for $\hat{J}^{(i)}$ and $\nabla \hat{J}^{(i)}$ is to prevent abrupt changes in these functions, which could pose challenges for gradient-based optimization algorithms. The remark after \Cref{def:three_kernels} guarantees this assumption for the Gaussian, the quadratic Matérn and the Wendland kernel of second order.}
\end{enumerate}
\end{assumption}


\subsection{The optimization subproblem, the approximated generalized Cauchy point and the formulation of the Hermite kernel trust-region algorithm}
\label{section_hermite_TR_kernel}
In the proposed \ac{hktr} algorithm (Algorithm \ref{kernel_TR_algorithm}) we construct the surrogate model that aims to approximate the objective function $J$ in the current \ac{tr} using Hermite kernel interpolation as introduced in \Cref{sect:Hermite}. An important part of every \ac{tr} algorithm is the computation of the next iterate $\mu^{(i+1)}$ by minimizing the constructed surrogate model $\hat{J}^{(i)}$ in the current \ac{tr}. 
In classical \ac{tr} methods, the optimization subproblem is typically solved within a ball centered at the current iterate, defined in \eqref{eqn:tr_sphere}. This constraint reflects the assumption that the surrogate model is only reliable in a small neighborhood of $\mu^{(i)}$. \\

In the data-driven Hermite kernel interpolation framework, a feasible neighborhood can be defined in a more sophisticated way, using the upper bound on the (Hermite) kernel interpolation error stated in \eqref{eqn:def:interpolation error}. Defining for $\mu \in \mathcal{P}$ \begin{align}
\label{eq:eta_upper_bound}
    \eta^{(i)}(\mu):= \Vert f \Vert_{\mathcal{H}_k(\mathcal{P})} P_{M^{(i)}}(\mu)
\end{align}  yields the following advanced (adv) definition for the \ac{tr}:
\begin{align*}
B^{(i)}_{\textnormal{adv}} := \left\{ \mu \in \mathcal{P} \;\middle|\; \frac{\eta^{(i)}(\mu)}{\hat{J}^{(i)}(\mu)} \leq \delta^{(i)} \right\}.
\end{align*}
This formulation allows feasible points to lie anywhere in $\mathcal{P}$ as long as the (relative) upper bound on the interpolation error remains controlled, thereby potentially enlarging the \ac{tr} and allowing the surrogate to be exploited more effectively. We summarize this in the following definition.

\begin{definition}(Optimization subproblem) \\
Define the optimization subproblem as 
\begin{align}
\label{suboptimzation_problem_abstract}
\min_{\mu \in \mathcal{P}} \hat{J}^{(i)}(\mu) \textnormal{ s.t. } c^{(i)}(\mu) \geq 0.
\end{align}
For the constraint $c^{(i)}$ we pose \begin{align}
\label{suboptimzation_problem_advanced}
c^{(i)}(\mu) := \delta^{(i)} - \frac{\eta^{(i)}(\mu)}{\hat{J}^{(i)}(\mu)} =  \delta^{(i)} - \frac{P_{M^{(i)}}(\mu)\Vert J\Vert_{\mathcal{H}_k(\mathcal{P})}}{\hat{J}^{(i)}(\mu)}.
\end{align}
\end{definition}

To solve the optimization subproblem (\ref{suboptimzation_problem_abstract}) we employ a gradient descent method.  These algorithms examine the surrogate model $\hat{J}^{(i)}$ along a descent direction $p^{(i)}$ within the current \ac{tr} $B_{\textnormal{adv}}^{(i)}$. Commonly the first search direction $p^{(i)}$ is chosen as $ - \nabla J^{(i)}(\mu^{(i)})$. It seems reasonable that we obtain a good reduction of the surrogate model $\hat{J}^{(i)}$, if we move in this direction as long as the function value of the surrogate model still decreases. More formally, we want to compute the minimum of $\hat{J}^{(i)}$ by a line search (ls) along the following line: \begin{align}
\label{eqn:lsForConvergence}
\mu^{(i)}_{\textnormal{ls}} := \left\{ \mu \in B_{\textnormal{adv}}^{(i)} \;|\; \mu:= \mu^{(i)} + \alpha p^{(i)} ,\; \alpha \geq 0 \right\}. 
\end{align}
Computing the exact value $\alpha_{\textnormal{min,ls}}^{(i)}$ corresponding to the value that minimizes $\hat{J}^{(i)}$ on the line $\mu^{(i)}_{\textnormal{ls}}$ might be difficult for a general model $\hat{J}^{(i)}$. Therefore, a backtracking (bt) strategy to find a point that achieves a sufficient decrease of the surrogate model $\hat{J}^{(i)}$ is applied: Find the smallest non-negative integer $j = j^{(i)}_{\textnormal{AGC}} \in \mathbb{N}_0$ s.t. \begin{align*}
\mu^{(i)}(j) := \mu^{(i)} + \kappa_{\textnormal{bt}}^j  p^{(i)}
\end{align*}
satisfies the Armijo (arm) condition \begin{align}
\hat{J}^{(i)}(\mu^{(i)}(j)) - \hat{J}^{(i)}(\mu^{(i)}) &\leq - \kappa_{\textnormal{arm}} \left\Vert\nabla \hat{J}^{(i)}(\mu^{(i)})\right\Vert \left\Vert\mu^{(i)} - \mu^{(i)}(j)\right\Vert  \cos \Phi^{(i)}, \label{subproblem_armijo_cond} \\
c^{(i)}(\mu^{(i)}(j)) &\geq 0, \label{subproblem_constraint_cond}
\end{align}
where $\kappa_{\textnormal{bt}} \in (0,1)$ and $\kappa_{\textnormal{arm}} \in (0, 0.5)$ are given constants. Typical values according to literature are $\kappa_{\textnormal{arm}} = 10^{-4}$ and $\kappa_{\textnormal{bt}} = 0.5$, compare \cite{ct_kelly_99}. Here $\Phi^{(i)}$ denotes the angle between $ - \nabla \hat{J}^{(i)}(\mu^{(i)})$ and $p^{(i)}$. We can now define the \ac{agc} point as \begin{align*}
\mu_{\textnormal{AGC}}^{(i)} := \mu^{(i)}(j_{\textnormal{AGC}}^{(i)}).
\end{align*}
The \ac{agc} point $\mu^{(i)}_{\textnormal{AGC}} =: \mu^{(i,1)}$ defines the first successful iterate of the gradient descent method. The algorithm now proceeds using the next descent direction and again performing the Armijo backtracking search: For $l \in \mathbb{N}$ define $\mu^{(i,l+1)} := \mu^{(i,l)}(j^{(i,l)})$, where $j^{(i,l)} \in \mathbb{N}_0$ is the first non-negative integer, s.t. (\ref{subproblem_armijo_cond}) and (\ref{subproblem_constraint_cond}) hold using $\mu^{(i,l)}$ instead of $\mu^{(i)}$. As gradient descent method we utilize the \ac{bfgs} algorithm, cf. \cite{Broyden,Fletscher, Goldfarb, Shanno} in the experiments presented in Section \ref{chapter_numerical_examples}. In this case, the search direction gets updated via the \ac{bfgs} update formula. Figure \ref{figure_tr_points_explanation} visualizes the relation between the current iterate $\mu^{(i)}$, the \ac{agc} point $\mu^{(i)}_{\textnormal{AGC}}$ and the minimizer of $\hat{J}^{(i)}$ in the current \ac{tr}, which we call $\mu_{\textnormal{min}}^{(i)}$. Note that for the sake of visualization the \ac{tr} is displayed as a ball. \\
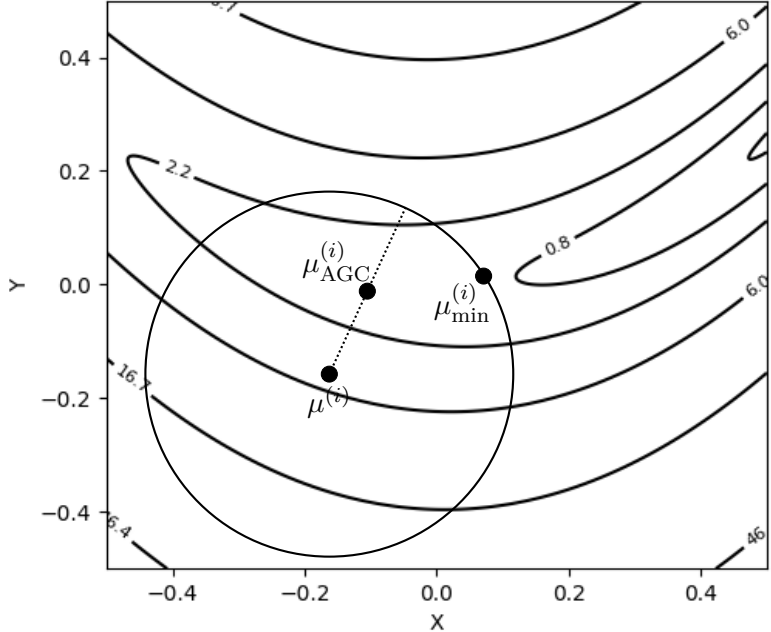
\begin{figure}[ht!]
    \centering
    \includegraphics[scale=0.8]{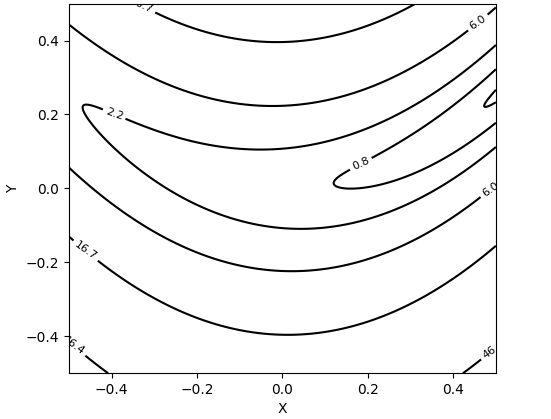}
    \begin{tikzpicture}[remember picture, overlay]
        \node[text=black] at (-7, 3.13) {$\mu^{(i)}$};
        \node[text=black] at (-6.9, 5) {$\mu_{\textnormal{AGC}}^{(i)}$};
        \node[text=black] at (-5.25, 4.45) {$\mu_{\textnormal{min}}^{(i)}$};
        \draw[black, thick] (-7, 3.5) circle (2.42cm);
        \filldraw[black] (-7, 3.5) circle (3pt);
       \filldraw[black] (-6.5, 4.6) circle (3pt);
        \filldraw[black] (-4.97, 4.8) circle (3pt);
        \draw[densely dotted, thick] (-7, 3.5) -- (-6, 5.7);
    \end{tikzpicture}
    \caption{The current iterate $\mu^{(i)}$, the approximate Cauchy point $\mu^{(i)}_{\textnormal{AGC}}$ and the model minimizer $\mu_{\textnormal{min}}^{(i)}$ on a contour plot of the Rosenbrock function defined as 
    $f(x,y) := (1-x)^2 + 100(y-x^2)^2$.}
    \label{figure_tr_points_explanation}
\end{figure} 

In a practical implementation, we propose the following termination criteria for the optimization subproblem when iterating over $l$ for a fixed $i$:
\begin{align}
\label{termination_subproblem}
\left\Vert\nabla \hat{J}^{(i)}(\mu^{(i,l)}) \right\Vert_\infty \leq \tau_{\textnormal{sub}} \quad \textnormal{ or }  \quad \beta_2 \delta^{(i)} \leq \frac{\eta^{(i)}(\mu^{(i,l)})}{\hat{J}^{(i)}(\mu^{(i,l)})} \leq \delta^{(i)}.
\end{align}
Here typically $\tau_{\textnormal{sub}}\ll 1$ and $\beta_2 \in (0,1)$, generally close to one according to \cite{Keil_2021}.  The second termination criterion avoids that excessive time is spent near the boundary of the \ac{tr} $B^{(i)}_{\textnormal{adv}}$. This is important because the kernel surrogate model $\hat{J}^{(i)}$ is likely to provide a poor approximation in that region, as discussed in \cite{Keil_2021, Qian2017ACT}. \\


To verify if a solution $\mu^{(i+1)} := \mu^{(i,l^{(i)})}$, with $l^{(i)}$ being the amount of iterates the gradient descent algorithm required to solve the subproblem,  yields a sufficient decrease of the kernel surrogate model $\hat{J}^{(i)}$, we make use of the \ac{agc} point $\mu_{\textnormal{AGC}}^{(i)}$.
This point will serve the purpose of evaluating whether a new iterate $\mu^{(i+1)}$ achieves a satisfactory reduction in the objective function $J$ and can consequently be accepted. We adapt the ideas from \citep[Section 4.2]{Keil_2021} and \cite[Section 3.2.3]{YueMeerbergen} to the Hermite kernel setting for the remainder of this section. 

\begin{condition}(Sufficient decrease condition) \\
The sufficient decrease condition for the \ac{hktr} algorithm is \begin{align}
\label{error_aware_sufficient_decrease_condition}
J(\mu^{(i+1)}) \leq \hat{J}^{(i)}(\mu_{\textnormal{AGC}}^{(i)}). 
\end{align}
\end{condition}
The underlying motivation for this condition is that the \ac{agc} point $\mu_{\textnormal{AGC}}^{(i)}$ is expected to yield a decrease in the surrogate objective $\hat{J}^{(i)}$ relative to the current iterate $\mu^{(i)}$. This decrease will be formalized and rigorously proven in \Cref{section_convergence_theory_kernel_TR}.
If (\ref{error_aware_sufficient_decrease_condition}) is satisfied, we accept $\mu^{(i+1)}$, build the next kernel surrogate model $\hat{J}^{(i+1)}$ and continue to formulate the $(i+1)$-th optimization subproblem. However, checking (\ref{error_aware_sufficient_decrease_condition}) is computationally expensive, as we have to evaluate $J(\mu^{(i+1)})$. If the check \eqref{error_aware_sufficient_decrease_condition} fails, the point $\mu^{(i+1)}$ gets rejected and we wasted computational time. To prevent this scenario from occurring, we now establish sufficient and necessary conditions for (\ref{error_aware_sufficient_decrease_condition}) that do not require the evaluation $J(\mu^{(i+1)})$. This has the potential to significantly reduce the computational cost of the \ac{hktr} algorithm.

\begin{lemma}
\label{theorem_when_to_accept_iterate}
Using the definition of the upper bound on the interpolation error from \eqref{eq:eta_upper_bound}, \begin{enumerate}
\item a sufficient condition for (\ref{error_aware_sufficient_decrease_condition}) is \begin{align}
\label{eqn_requirement_acceptance}
\hat{J}^{(i)}(\mu^{(i+1)}) + \eta^{(i)}(\mu^{(i+1)}) \leq \hat{J}^{(i)}(\mu^{(i)}_{\textnormal{AGC}}),
\end{align}
\item a necessary condition for (\ref{error_aware_sufficient_decrease_condition}) is 
\begin{align}
\label{eqn_requirement_rejection}
\hat{J}^{(i)}(\mu^{(i+1)}) - \eta^{(i)}(\mu^{(i+1)}) \leq \hat{J}^{(i)}(\mu^{(i)}_{\textnormal{AGC}}).
\end{align}
\end{enumerate}
\end{lemma}
\begin{proof}
Sufficient condition: Using the definition of $\eta^{(i)}$ in \eqref{eq:eta_upper_bound} yields \begin{align*}
     J(\mu^{(i+1)}) \leq  \hat{J}^{(i)}(\mu^{(i+1)}) + \eta^{(i)}(\mu^{(i+1)}).
\end{align*}
by utilizing the inverse triangle inequality. Therefore, \eqref{eqn_requirement_acceptance} implies \eqref{error_aware_sufficient_decrease_condition}. Necessary condition: Similarly, it holds \begin{align*}
 J(\mu^{(i+1)}) \geq \hat{J}^{(i)}(\mu^{(i+1)}) - \eta^{(i)}(\mu^{(i+1)}).
\end{align*}
Therefore, if \eqref{error_aware_sufficient_decrease_condition} holds, also \eqref{eqn_requirement_rejection} is satisfied. 
\end{proof}

Based on these considerations, we propose the following computational procedure instead of directly checking (\ref{error_aware_sufficient_decrease_condition}): \begin{enumerate}
\item Check (\ref{eqn_requirement_acceptance}). If the condition holds, we accept $\mu^{(i+1)}$ as the next iterate. 
\item If (\ref{eqn_requirement_acceptance}) fails, we check (\ref{eqn_requirement_rejection}). If this condition holds, we reject the proposed iterate $\mu^{(i+1)}$ and solve the optimization subproblem (\ref{suboptimzation_problem_abstract}) again, using a shrinked \ac{tr} radius $\delta^{(i)}$.
\item If neither (\ref{eqn_requirement_acceptance}) nor (\ref{eqn_requirement_rejection}) hold, we have to check (\ref{error_aware_sufficient_decrease_condition}) directly. This is computationally expensive and we try to avoid this case if possible. 
\end{enumerate}
These three cases are reflected in lines 4, 7 and 11 of the proposed \ac{hktr} algorithm (Algorithm \ref{kernel_TR_algorithm}) where all the implementation details discussed so far are comprised.

\begin{algorithm}[ht!]
\caption{Hermite kernel trust-region (HKTR) algorithm}\label{kernel_TR_algorithm}
\KwIn{Objective function $J$, initial iterate $\mu^{(0)}$, initial \ac{tr} radius $\delta^{(0)}$, maximum iterations of the \ac{tr} algorithm $i_{\textnormal{max}}$, stopping tolerances for the optimization subproblem $\tau_{\textnormal{sub}}$ and $\tau_J$, maximum iterations for the optimization subproblem $l_{\textnormal{max}}$, backtracking step $\kappa_{\textnormal{bt}}$, Armijo constant $\kappa_{\textnormal{arm}}$, \ac{foc} tolerance $\tau_{\textnormal{FOC}}$, \ac{tr} shrinking factor $\beta_1$, safeguard for the \ac{tr} boundary condition $\beta_2$, tolerance for enlarging the \ac{tr} radius $\xi$.}
Set $i:=0$ and $LoopFlag := True$. \\
\While{$i \leq i_{\textnormal{max}}$ and LoopFlag = True}{
	Compute $\mu^{(i+1)}$ as solution of the optimization subproblem (\ref{suboptimzation_problem_abstract}) using \ac{bfgs} with the termination criteria specified in (\ref{termination_subproblem}). This algorithm also returns $\mu^{(i)}_{\textnormal{AGC}}$ as its first successful iterate. \label{go_back_to_optim_subproblem} \\
	\uIf{$\hat{J}^{(i)}(\mu^{(i+1)}) + \eta^{(i)}(\mu^{(i+1)}) \leq \hat{J}^{(i)}(\mu^{(i)}_{\textnormal{AGC}})$}{
		Accept $\mu^{(i+1)}$, build the new kernel surrogate model $\hat{J}^{(i+1)}$ around $\mu^{(i+1)}$. \\
		Compute $\rho^{(i)}$ according to (\ref{check_sufficient_decrease_of_model}) and update the \ac{tr} radius according to \eqref{updating_tr_radius}.
	}
	\uElseIf{$\hat{J}^{(i)}(\mu^{(i+1)}) - \eta^{(i)}(\mu^{(i+1)}) > \hat{J}^{(i)}(\mu^{(i)}_{\textnormal{AGC}})$} {
		Reject the new iterate $\mu^{(i+1)}$, shrink the \ac{tr} radius: $\delta^{(i)} := \beta_1 \delta^{(i)}$ and go back to line \ref{go_back_to_optim_subproblem} without increasing $i$.
	}	
	\Else{
		Evaluate $J(\mu^{(i+1)})$, $\nabla J (\mu^{(i+1)})$ and build the new kernel surrogate model $\hat{J}^{(i+1)}$ including the data for $\mu^{(i+1)}$. \\
		\uIf{$J(\mu^{(i+1)}) \leq \hat{J}^{(i)}(\mu^{(i)}_{\textnormal{AGC}})$ \label{line_kernel_algo_sufficient_decrease_check}}{
			Accept $\mu^{(i+1)}$. \\
			Compute $\rho^{(i)}$ according to (\ref{check_sufficient_decrease_of_model}) and update the \ac{tr} radius according to \eqref{updating_tr_radius}.
		}
		\Else{
			Reject $\mu^{(i+1)}$, set $\hat{J}^{(i)} := \hat{J}^{(i+1)}$ (i.e., keep the updated model) shrink the \ac{tr} radius: $\delta^{(i)} := \beta_1 \delta^{(i)}$ and go back to line \ref{go_back_to_optim_subproblem} without increasing $i$.
		}
	}
	\If{$\left\Vert\nabla \hat{J}^{(i+1)}(\mu^{(i+1)})\right\Vert_\infty \leq \tau_{\textnormal{FOC}}$ or $\hat{J}^{(i)}_{\textnormal{diff}} \leq \tau_J$}{
		$LoopFlag := False$. 
	}
	$i := i + 1$.
}
\KwOut{Sequence of iterates $\left\{ \mu^{(i)} \right\}$, sequence of function values $\left\{ J(\mu^{(i)}) \right\}$,  sequence of \ac{foc} conditions $\left\{ \left\Vert \nabla J(\mu^{(i)})\right\Vert_\infty \right\}$.}
\end{algorithm}


\subsection{Convergence analysis}
\label{section_convergence_theory_kernel_TR}
To establish convergence of the kernel \ac{tr} algorithm 
(Algorithm~\ref{kernel_TR_algorithm}), it is assumed that the algorithm 
generates an infinite sequence of iterates 
$\left\{\mu^{(i)}\right\}_{i \in \mathbb{N}_0}$. 
In Theorem~\ref{auxilary_statements_from_YM}, a lower bound is derived for 
the decrease in \(\hat{J}^{(i)}\) achieved by \(\mu^{(i)}_{\textnormal{AGC}}\). 
A key assumption in the proof of Theorem~\ref{auxilary_statements_from_YM} 
is the Hölder continuity of \(c^{(i)}\), defined in the constraint of the 
subproblem \eqref{suboptimzation_problem_abstract}, with Hölder exponent 
$\alpha_{\textnormal{Höl}} = \sfrac{1}{2}$. The following theorem demonstrates 
that this requirement is satisfied for \(c^{(i)}\) as defined in 
\eqref{suboptimzation_problem_advanced}.

\begin{lemma}
\label{lemma:c_hölder}
Let the conditions of Theorem \ref{theorem_P_hölder} be satisfied. Then $c^{(i)}$, defined in (\ref{suboptimzation_problem_advanced}) as \begin{align*}
c^{(i)}(\mu) = \delta^{(i)} - \frac{\eta^{(i)}(\mu)}{\hat{J}^{(i)}(\mu)} = \delta^{(i)} - \frac{P_{M^{(i)}}(\mu) \Vert J \Vert_{\mathcal{H}_k(\mathcal{P})}}{\hat{J}^{(i)}(\mu)}, 
\end{align*}
is Hölder continuous with the Hölder exponent $\alpha_{\textnormal{Höl}} = \sfrac{1}{2}$, i.e., it exists $C_c^{(i)} \geq 0$ s.t. \begin{align*}
\left|c^{(i)}(\mu) - c^{(i)}(\tilde{\mu})\right| \leq C_c^{(i)} \Vert\mu - \tilde{\mu}\Vert^{\frac{1}{2}} \quad \forall \; \mu, \tilde{\mu} \in \mathcal{P}.
\end{align*} 
\end{lemma}

\begin{proof}
 Theorem \ref{theorem_P_hölder} states that $P_{M^{(i)}}$ is Hölder continuous with $\alpha_{\textnormal{Höl}} = \sfrac{1}{2}$. According to Assumption \ref{assumption_list} c) and e), $\hat{J}^{(i)}$ is uniformly bounded away from zero and Lipschitz continuous. Thus, the fraction \begin{align*}
\frac{P_{M^{(i)}}(\mu)}{\hat{J}^{(i)}(\mu)}
\end{align*} is also Hölder continuous with $\alpha_{\textnormal{Höl}} = \sfrac{1}{2}$. Since $\Vert J \Vert_{\mathcal{H}_k(\mathcal{P})}$ as well as $\delta^{(i)}$ are constant values, we obtain the desired result. \end{proof}

The next theorem is based on \citep[Theorem 3.2]{YueMeerbergen} .
The theorem states a lower bound for the decrease in $\hat{J}^{(i)}$ achieved by the \ac{agc} point $\mu^{(i)}_{\textnormal{AGC}}$ and is the key result in the convergence analysis. 

\begin{theorem}
\label{auxilary_statements_from_YM}
Let the assumptions of Lemma \ref{lemma:c_hölder} be satisfied, s.t. $c^{(i)}$ is Hölder continuous with exponent $\alpha_{\textnormal{Höl}} = \sfrac{1}{2}$ and Hölder constant $C_c^{(i)} > 0$.
Further, let Assumption \ref{assumption_list} e)  hold, i.e., $\nabla \hat{J}^{(i)}$ is Lipschitz continuous, so there exists $ C_{\nabla \hat{J}}^{(i)} > 0$ s.t. \begin{align*}
\left\Vert\nabla \hat{J}^{(i)}(\mu) - \nabla \hat{J}^{(i)}(\tilde{\mu})\right\Vert \leq C_{\nabla \hat{J}}^{(i)} \Vert \mu - \tilde{\mu}\Vert \quad \forall \; \mu, \tilde{\mu} \in \mathcal{P}.
\end{align*}
Let furthermore $\Phi^{(i)} < \frac{\pi}{2}$, $\kappa_{\textnormal{arm}} \in (0,1)$, $c^{(i)}(\mu^{(i)}) > 0$. Then, we obtain the following result: 
A lower bound for the decrease in $\hat{J}^{(i)}$ achieved by the \ac{agc} point $\mu_{\textnormal{AGC}}^{(i)}$ is given by \begin{align}
\label{theorem3.2_eqn315}
\hat{J}^{(i)} (\mu^{(i)}) - \hat{J}^{(i)} (\mu_{\textnormal{AGC}}^{(i)}) \geq \left(\kappa_{\textnormal{arm}} \cos \Phi^{(i)} \right) \left\Vert\nabla \hat{J}^{(i)} (\mu^{(i)}) \right\Vert \min \left\{ \kappa_{\nabla \hat{J}}^{(i)} \left\Vert\nabla \hat{J}^{(i)} (\mu^{(i)})\right\Vert, \kappa_{\textnormal{bt}} \frac{\left(c^{(i)}(\mu^{(i)})\right)^2}{\left(C_c^{(i)} \right)^2} \right\},
\end{align}
where $\kappa_{\nabla \hat{J}}^{(i)} := \min \left\{ 1, \frac{\kappa_{\textnormal{bt}}(1 - \kappa_{\textnormal{arm}})\cos \Phi^{(i)}}{C_{\nabla \hat{J}}^{(i)}} \right\}$.
\end{theorem}

\begin{proof}
    See Appendix \ref{appendix_proof_YM}.
\end{proof}

The following theorem is an adapted version of \citep[Theorem 3.3]{YueMeerbergen}. It shows the convergence of the objective function $J$, if the sufficient decrease condition (\ref{error_aware_sufficient_decrease_condition}) is satisfied for all iterations $i$. Furthermore, it requires uniformity in $i$ for the Hölder and Lipschitz constants defined in Theorem \ref{auxilary_statements_from_YM}, which is satisfied, as the Hölder constant in Theorem \ref{theorem_P_hölder} and the Lipschitz constant in Theorem \ref{thm:LipschitzDerivKernel} are independent of the number of interpolation points, i.e., independent of the iteration $i$.  

\begin{theorem}
\label{theorem_convergence_kernel_tr}
Assume that all conditions of Theorem \ref{auxilary_statements_from_YM} hold and $\Phi^{(i)} \leq \Phi < \frac{\pi}{2}$, $c^{(i)}(\mu^{(i)}) \geq c_l > 0$, $ 0 < C_c^{(i)} \leq C_c$, $0 < C_{\nabla \hat{J}}^{(i)} < C_{\nabla \hat{J}}$ for all $i$. Then, if the sufficient decrease condition (\ref{error_aware_sufficient_decrease_condition}) holds for all iterations $i$, we get: \begin{align*}
\lim_{i \rightarrow \infty} \left\Vert\nabla J(\mu^{(i)}) \right\Vert = 0.
\end{align*}
\end{theorem}

\begin{proof}
According to $J(\mu^{(i+1)}) = \hat{J}^{(i+1)}(\mu^{(i+1)})$, (\ref{error_aware_sufficient_decrease_condition}) and \eqref{theorem3.2_eqn315} we get \begin{align*}
\hat{J}^{(i)}(\mu^{(i)}) &- \hat{J}^{(i+1)}(\mu^{(i+1)})  \\
&\geq \hat{J}^{(i)}(\mu^{(i)}) - \hat{J}^{(i)}(\mu_{\textnormal{AGC}}^{(i)}) \\
&\geq \left(\kappa_{\textnormal{arm}} \cos \Phi^{(i)} \right) \left\Vert\nabla \hat{J}^{(i)} (\mu^{(i)}) \right\Vert \min \left\{ \kappa_{\nabla \hat{J}}^{(i)} \; \left\Vert\nabla \hat{J}^{(i)} (\mu^{(i)}) \right\Vert, \kappa_{\textnormal{bt}} \frac{\left(c^{(i)}(\mu^{(i)})\right)^2}{\left(C_c^{(i)} \right)^2} \right\}.
\end{align*}
Using summation, we end up with \begin{align*}
\hat{J}^{(0)}(\mu^{(0)}) &- \hat{J}^{(m)}(\mu^{(m)})  
\\ &\geq \kappa_{\textnormal{arm}} \sum_{i=0}^{m-1} \cos \Phi^{(i)} \left\Vert\nabla \hat{J}^{(i)} (\mu^{(i)}) \right\Vert \min \left\{ \kappa_{\nabla \hat{J}}^{(i)} \; \left\Vert\nabla \hat{J}^{(i)} (\mu^{(i)}) \right\Vert, \kappa_{\textnormal{bt}} \frac{\left(c^{(i)}(\mu^{(i)})\right)^2}{\left(C_c^{(i)}\right)^2} \right\},
\end{align*}
for all $m \in \mathbb{N}$, which can be seen via induction. Since \begin{align*}
\kappa_{\nabla \hat{J}}^{(i)} = \min \left\{ 1,  \frac{\kappa_{\textnormal{bt}}(1 - \kappa_{\textnormal{arm}}) \cos \Phi^{(i)}}{C_{\nabla \hat{J}}^{(i)}} \right\} \geq \min \left\{ 1, \frac{\kappa_{\textnormal{bt}}(1 - \kappa_{\textnormal{arm}}) \cos \Phi}{C_{\nabla \hat{J}}} \right\} =: \kappa_{\nabla\hat{J}},
\end{align*}
we get \begin{align*}
\hat{J}^{(0)}(\mu^{(0)}) &- \hat{J}^{(m)}(\mu^{(m)}) \\
&\geq (\kappa_{\textnormal{arm}} \cos \Phi) \sum_{i=0}^{m-1} \left\Vert \hat{J}^{(i)}(\mu^{(i)}) \right\Vert \min \left\{ \kappa_{\nabla \hat{J}} \; \left\Vert\hat{J}^{(i)}(\mu^{(i)}) \right\Vert, \kappa_{\textnormal{bt}} \frac{c_l^2}{C_c^2} \right\}.
\end{align*}
Now we prove $\lim_{i \rightarrow \infty} \left\Vert\nabla J(\mu^{(i)}) \right\Vert = 0$ by contradiction. Assume there exists an $\epsilon \in \left(0, \left\Vert\nabla \hat{J}^{(0)}(\mu^{(0)}) \right\Vert \right)$ and a index-subsequence $\nu_j$ satisfying $\left\Vert \nabla \hat{J}^{(\nu_j)} (\mu^{(\nu_j)}) \right\Vert \geq \epsilon$ for all $j \in \mathbb{N}$ with $\nu_0 = 0$. Applying $\lim_{j \rightarrow \infty}$ on both sides of the previous inequality yields \begin{align*}
\lim_{j \rightarrow \infty} \hat{J}^{(0)} (\mu^{(0)}) - \hat{J}^{(\nu_j)}(\mu^{(\nu_j)}) \geq \lim_{j \rightarrow \infty} \kappa_{\textnormal{arm}} \cos \Phi \sum_{m=0}^{j-1} \epsilon \min \left\{ \kappa_{\nabla \hat{J}} \; \epsilon, \kappa_{\textnormal{bt}} \frac{c_l^2}{C_c^2} \right\} = + \infty,
\end{align*}
contradicting the fact that $\hat{J}^{(0)}(\mu^{(0)})$ is finite and $\hat{J}^{(i)}(\mu^{(i)})$ is bounded from below. Therefore, \begin{align*}
    \lim_{i \rightarrow \infty} \left\Vert\nabla \hat{J}^{(i)}(\mu^{(i)}) \right\Vert = \lim_{i \rightarrow \infty} \left\Vert\nabla J(\mu^{(i)}) \right\Vert = 0.
\end{align*}   
\end{proof}

All assumptions of Theorem \ref{theorem_convergence_kernel_tr} are not very restrictive except that (\ref{error_aware_sufficient_decrease_condition}) has to be satisfied for all iterations $i$. In Section \ref{section_hermite_TR_kernel}, we have already discussed how to deal with the scenario where (\ref{error_aware_sufficient_decrease_condition}) is not satisfied and proposed sufficient and necessary conditions that should be utilized instead of (\ref{error_aware_sufficient_decrease_condition}). 

\subsection{Parameter constrained optimization problems}
\label{section:projected_hktr}
Until now, we focused on the optimization problem using the feasible set $\mathcal{P} = \mathbb{R}^p$. In this section, we comment on the required changes if we restrict the iterates to subsets  $\mathcal{P} \subset \mathbb{R}^p$ of the form \label{section_constrained_optimization}
\begin{align*}
\mathcal{P} := \left\{ \mu \in \mathbb{R}^p \; | \; \mu_a \leq \mu \leq \mu_b \right\} \subset \mathbb{R}^p.
\end{align*}
Here, we have $\mu_a, \mu_b \in (\mathbb{R} \cup \{ \pm \infty \})^p $ and $\leq$ should be understood component-wise. This is commonly referred to as box-constraints. The primary distinction compared to the unconstrained scenario is that we must ensure that all calculated iterates, both in the \ac{hktr} algorithm and when solving the optimization subproblem, remain within the specified parameter set $\mathcal{P}$. In order to describe this rigorously, we define a projection map that maps $\mu$ to the nearest point (measured in the Euclidean norm) in $\mathcal{P}$.

\begin{definition}(Projection map) \\
We define the projection map $\Pi_{\mathcal{P}}: \mathbb{R}^p \rightarrow \mathcal{P}$ as \begin{align*}
(\Pi_{\mathcal{P}}(\mu))_m := \begin{cases}
(\mu_a)_m \quad \textnormal{ if } \mu_m \leq (\mu_a)_m \\
(\mu_b)_m \quad \textnormal{ if } \mu_m \geq (\mu_b)_m \\
(\mu)_m \quad \; \; \textnormal{ otherwise }
\end{cases}
\qquad \forall \; m=1,...,p.
\end{align*} 
\end{definition}

We also introduce \begin{align*}
\mu^{(i,l)}(j) := \Pi_{\mathcal{P}}(\mu^{(i,l)} + \kappa_{\textnormal{bt}}^j p^{(i,l)}) \textnormal{ for } j \geq 0, 
\end{align*}
in order to guarantee, that all iterates of the \ac{bfgs} algorithm and the Armijo backtracking search also lie within $\mathcal{P}$. We have to reformulate the termination criteria for the optimization subproblem as well as the one for the \ac{hktr} algorithm (Algorithm \ref{kernel_TR_algorithm}) as \begin{align*}
\left\Vert \mu^{(i,l)} - \Pi_{\mathcal{P}} \left( \mu^{(i,l)} - \nabla \hat{J}^{(i)} (\mu^{(i,l)}) \right) \right\Vert_\infty \leq \tau_{\textnormal{sub}},
\end{align*}
respectively 
\begin{align*}
\left\Vert \mu^{(i)} - \Pi_{\mathcal{P}} \left( \mu^{(i)} - \nabla \hat{J}^{(i)} (\mu^{(i)}) \right) \right\Vert_\infty \leq \tau_{\textnormal{FOC}}.
\end{align*}
The convergence proof of this projected version of the \ac{hktr}, which we will refer to as the \ac{phktr} algorithm, follows identical to the argumentation presented in Section \ref{section_convergence_theory_kernel_TR} and relies on the Lipschitz continuity of the projection map $\Pi_\mathcal{P}$ with Lipschitz constant $C=1$. We refer to \citep[Section 4.2]{Keil_2021}, which outlines a convergence proof based on such projected quantities. 


\section{Numerical examples}
\label{chapter_numerical_examples}
In this section, we apply the \ac{phktr} algorithm (Algorithm \ref{kernel_TR_algorithm}) to solve three optimization problems. The first one, a 1D problem, is a toy example specifically designed for the Gaussian kernel. The other two problems are \ac{pde}-constrained optimization problems, considered in 2D and 12D, respectively. The code for this section with the results of the numerical experiments presented below can be found on \texttt{GitHub}\footnote{See \url{https://github.com/ullmannsven/A-Trust-Region-framework-for-optimization-using-Hermite-kernel-surrogate-models}}.

\subsection{Setup and comparison}
In the following sections, the performance of the \ac{phktr} 
algorithm (Algorithm~\ref{kernel_TR_algorithm}) is compared with two 
methods from \texttt{scipy.optimize.minimize}, namely \texttt{trust-constr} 
and \texttt{L-BFGS-B}. Both of these methods accommodate box constraints on 
the parameter space and circumvent the need for an explicit Hessian 
computation. The \texttt{trust-constr} method belongs to the class of 
\ac{tr} algorithms, similar to 
Algorithm~\ref{kernel_TR_algorithm}, but employs a quadratic surrogate 
model similar to \eqref{quadratic_model_J}. In our context, it follows the 
implementation in \cite{trust-constr}, where the subproblem is solved via 
the sequential least squares quadratic programming method \cite{slsqp}. Meanwhile, \texttt{L-BFGS-B} is a popular 
choice for problems with box constraints that do not require explicit 
Hessian information \citep{lbfgsb_1,lbfgsb_2}. For all three methods, the same two termination criteria are used:
\[
\left\Vert \nabla J (\mu^{(i)}) \right\Vert_\infty
  \;\leq\; \tau_{\textnormal{FOC}}
\quad
\textnormal{or}
\quad
\frac{J(\mu^{(i)}) - J(\mu^{(i+1)})}
     {\max \left\{J(\mu^{(i)}),\,J(\mu^{(i+1)}),\,1 \right\}}
  \;\leq\;\tau_J.
\]
The specific values of \(\tau_{\textnormal{FOC}}\) and \(\tau_J\) are detailed 
for each experiment. Note that \texttt{trust-constr} does not allow a tolerance using $\tau_J$, so only $\tau_{\textnormal{FOC}}$ was used there. In each case, five random initial guesses 
\(\mu^{(0)} \in \mathcal{P}\) are generated to test 
\texttt{L-BFGS-B}, \texttt{trust-constr}, and the \ac{phktr} 
algorithm, where we use the same initial guesses for all three algorithms. A reference solution, used to evaluate accuracy, is computed via 
\texttt{L-BFGS-B} with stricter tolerances \(\tau_{\textnormal{FOC}}\) and 
\(\tau_J\). In all following sections, we measure the accuracy and efficiency of the \ac{phktr} algorithm by comparing the \ac{av} \ac{fom} evaluations until termination, the \ac{av} \ac{foc} condition $\Vert \nabla J(\cdot) \Vert$ at the last iteration and the \ac{av} relative error in $J$ to the reference solution, while testing different values for the kernel shape parameter $\varepsilon$. The remaining parameters of the \ac{phktr} algorithm are kept constant and we refer to the \texttt{GitHub} repository for the exact values.

\subsection{1D optimization problem} 
\label{subsection_numerical_examples_1D}
The 1D problem is designed as a tailored optimization problem to illustrate the application of the Gaussian kernel, as we choose the objective function as \begin{align*}
    J(\mu) = - \exp(-\mu^2) + 3 \exp(- 0.001 \mu^2).
\end{align*} 
Note that the numbers $3$ and $0.001$ in the definition of $J$ are chosen s.t. Assumption \ref{assumption_list} b) is satisfied. While evaluating $J$ is computationally inexpensive in this case, meaning there is no practical necessity to construct a surrogate model, we include this example to demonstrate the methodology and validate the approach in a controlled and straightforward scenario. We first demonstrate how to compute the \ac{rkhs}-norm for the objective function $J$, which can be done explicitly in this scenario. Following \cite[Theorem 10.12]{wendland_2004}, the \ac{rkhs}-norm - corresponding to a translation invariant \ac{spd} kernel $k$ with $\phi \in C(\mathbb{R}) \cap L^1(\mathbb{R})$ - of a univariate function $J \in L^2(\mathbb{R}) \cap C(\mathbb{R})$, s.t $\frac{\mathcal{F}(J)}{\sqrt{\mathcal{F}(\phi)}} \in L^2(\mathbb{R})$, can be computed via \begin{align} \label{eq:integral_rkhsnorm}
    \Vert J\Vert_{\mathcal{H}_{k}}^2 = \frac{1}{\sqrt{2 \pi}} \int_{\mathbb{R}} \frac{|\mathcal{F}(J)(\omega)|^2}{\mathcal{F}(\phi)(\omega)} d\omega,
\end{align}
where $\mathcal{F}$ denotes the Fourier transformation. We refer to Appendix \ref{appendix_rhhsnorm} for the exact computation and note that the \ac{rkhs}-norm is well defined for $\varepsilon^2 > \sfrac{1}{2}$, whereas for $\varepsilon^2 \leq \sfrac{1}{2}$ the integral in \eqref{eq:integral_rkhsnorm} diverges. With this knowledge at hand, we start to solve the optimization problem. As the minimizer for $J$ is $\mu^* = 0 $ with $J(\mu^*) = 2$, we choose as parameter set $\mathcal{P} := [-2,2]$, which is symmetric around the optimal value and $J$ only has one extrema (at $\mu^*=0$) in this interval. Until the end of this section, the optimal parameter $\mu^*$ will serve as the reference solution, against which the accuracy and efficiency of the \ac{phktr} algorithm will be measured. As convergence criteria, we employ thresholds of \(\tau_{\textnormal{FOC}} = 10^{-7}\) 
for the \ac{foc} condition and \(\tau_J = 10^{-14}\) for the objective function.\\

The results for different kernel shape parameters $\varepsilon$ using the Gaussian kernel are displayed in Table \ref{table_results_1D_Gaussian}. The results demonstrate the relevance of the choice of $\varepsilon$. The best results were obtained for kernel shape parameters \(\varepsilon\) chosen close to, but strictly greater than, the lower admissible bound \(\varepsilon = \sfrac{1}{\sqrt{2}}\), which is not itself permitted. In this case, the \ac{phktr} algorithm converges slightly faster (in terms of \ac{fom} evaluations) than the two \texttt{scipy} algorithms, compare Table \ref{table_results_1D_compare}. Note that due to the simplicity of the objective function $J$, which is unimodal in $\mathcal{P}$, we can not expect major speedups with the proposed Algorithm \ref{kernel_TR_algorithm} compared to the \texttt{scipy} algorithms.
 
\begin{table}[ht!]
\centering
\small
\renewcommand{\arraystretch}{1.2}
\begin{tabular}{|c||c|c|c|c|c|}
\hline
kernel shape parameter $\varepsilon$ & \ac{av} \ac{fom} evaluations & \ac{av} \ac{foc} condition & \ac{av} relative error in $J$  \\
\hline 
\hline
$0.725$ & $5.6$ &  $1 \cdot 10^{-8}$ & $4 \cdot 10^{-17}$  \\
\hline 
$0.75$ & $6.0$ & $7 \cdot 10^{-8}$ & $2 \cdot 10^{-15}$  \\
\hline
$1.0$ & $6.6$ & $1 \cdot 10^{-8}$ & $9 \cdot 10^{-17}$  \\
\hline 
$2.0$ & $7.2$ & $1 \cdot 10^{-7}$ & $4 \cdot 10^{-15}$  \\
\hline
$10.0$ & $9.8$ & $1 \cdot 10^{-7}$ & $7 \cdot 10^{-15}$  \\
\hline 
\end{tabular}
\caption{Performance and accuracy of the \ac{phktr} algorithm using the Gaussian kernel to solve the 1D optimization problem for five optimization runs with randomly sampled initial parameters $\mu^{(0)} \in \mathcal{P}$.}
\label{table_results_1D_Gaussian}
\end{table}

\begin{table}[ht!]
\centering
\small
\renewcommand{\arraystretch}{1.2}
\begin{tabular}{|c||c|c|c|c|c|}
\hline
method & \ac{av} \ac{fom} evaluations & \ac{av} relative error in $J$  \\
\hline 
\hline
\ac{phktr} with $\varepsilon = 0.725$ & $5.6$ & $4 \cdot 10^{-17}$  \\
\hline 
\texttt{L-BFGS-B} & $6.2$  & $0$\\
\hline
\texttt{trust-constr} & $6.2$  & $0$\\
\hline 
\end{tabular}
\caption{Comparison of the \ac{phktr} algorithm using the Gaussian kernel to solve the 1D optimization problem for five optimization runs with randomly sampled initial parameters $\mu^{(0)} \in \mathcal{P}$ with the \texttt{L-BFGS-B} and \texttt{trust-constr} algorithm.}
\label{table_results_1D_compare}
\end{table}

\subsection{2D \ac{pde} constrained optimization problem}
\label{subsection_pde_contraint_num_example}
The second problem we address is formulated in the pyMOR (see \cite{pymor}) Tutorial: \textit{Model Order Reduction for \ac{pde}-constrained optimization problems}\footnote{\url{https://docs.pymor.org/2024-2-0/tutorial_optimization.html}}. We first provide a formulation of the optimization problem. We consider the domain $X := (-1,1) \times (-1,1)$, the parameter set $\mathcal{P} := [0.5, \pi] \times [0.5, \pi]$ and the parameter dependent, elliptic \ac{pde} with homogeneous Dirichlet boundary conditions \begin{align}
- \nabla \cdot (\lambda(x; \mu) \nabla u(x;\mu)) &= l(x) \quad \textnormal{ in } X \label{eqn:pde_constraint_2D}\\
u(x; \mu) &= 0  \quad \quad \textnormal{ on } \partial X \nonumber
\end{align}
with solution $u(\cdot, \mu) \in H_0^1(X)$, where $H_0^1(X)$ denotes the $L^2$-Sobolev space of order one with homogeneous Dirichlet boundary values. Here $x := \begin{bmatrix}
    x_1 & x_2
\end{bmatrix}^T \in X,  \; \mu := \begin{bmatrix}
    \mu_1 & \mu_2
\end{bmatrix}^T \in \mathcal{P}$ and \begin{align*}
l(x) &:= \sfrac{1}{2} \; \pi^2 \cos \left(\sfrac{1}{2} \; \pi x_1 \right) \cos \left(\sfrac{1}{2} \;\pi x_2 \right), \\ \lambda(x; \mu) &:= \theta_1(\mu)\lambda_1(x) + \theta_2(\mu) \lambda_2(x), \\
\theta_1(\mu) &:= 1.1 + \sin(\mu_1)\mu_2, \\
\theta_2(\mu) &:= 1.1 + \sin(\mu_2), \\
\lambda_1(x) &:= \chi_{X \setminus \omega}(x), \\
\lambda_2(x) &:= \chi_{\omega}(x), \\
\omega &:= \left( \left[-\sfrac{2}{3}, - \sfrac{1}{3} \right] \times \left[-\sfrac{2}{3}, - \sfrac{1}{3}\right] \right) \cup \left(\left[-\sfrac{2}{3}, - \sfrac{1}{3}\right] \times \left[\sfrac{1}{3}, \sfrac{2}{3}\right]\right).
\end{align*}
Here $\chi_A$ denotes the indicator function for the subset $A \subseteq X$. By multiplying \eqref{eqn:pde_constraint_2D} with a test function $v \in H_0^1(X)$, integrating over the domain $X$ and applying partial integration for the left-hand side, we obtain the primal equation \begin{align}
\label{primal_eqn_pymor}
\underbrace{\int_X \lambda(x; \mu) \nabla u(x; \mu) \cdot \nabla v(x) dx}_{=: a(u(x;\mu),v;\mu)} = \underbrace{\int_X l(x) v(x) dx }_{=: f(v)} \quad \forall \; v \in H_0^1(X).
\end{align}
Moreover, we consider an objective function depending on the solution $u(x;\mu)$ of the primal equation (\ref{primal_eqn_pymor}) \begin{align*}
J(\mu) := \theta_J(\mu) f(u(\cdot; \mu)) 
\end{align*}
with $\theta_J(\mu) := 1 + \frac{1}{5}(\mu_1 + \mu_2)$ for $\mu \in \mathcal{P}$. Every evaluation of the objective function $J$ involves a solution $u(x; \mu)$ of the primal equation (\ref{primal_eqn_pymor}). To obtain this solution, we utilize pyMOR's discretization toolkit, which allows to construct and solve parametrized \ac{fom}s. Figure \ref{figure:Objective_function2D} visualizes the objective function $J$ over the parameter set $\mathcal{P}$.  \\

\begin{figure}[ht!]
\centering 
\includegraphics[scale=0.5]{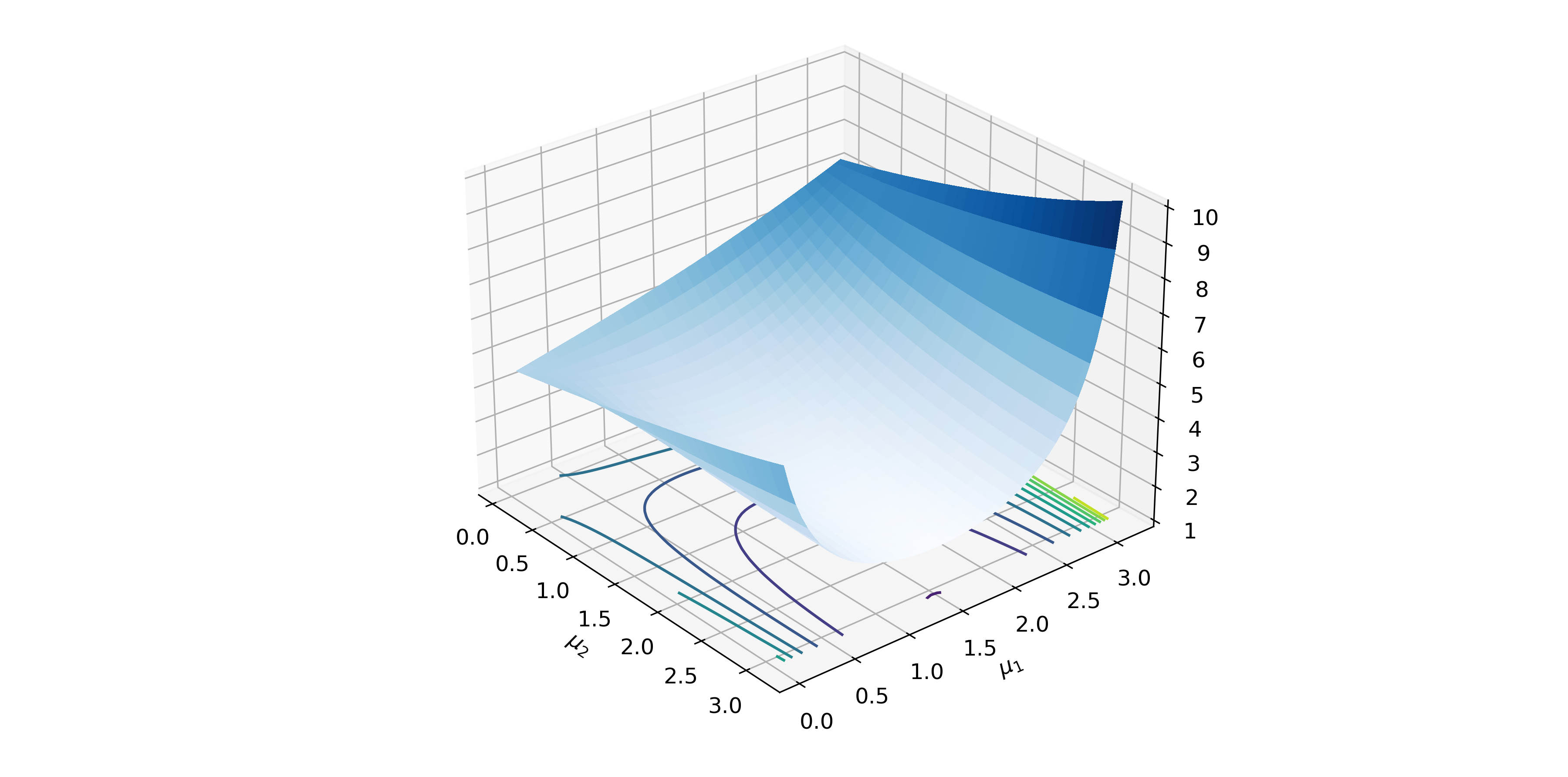}
\caption{Objective function $J$ over the parameter set $\mathcal{P}$}
\label{figure:Objective_function2D}
\end{figure}

As convergence criteria, we employ thresholds of \(\tau_{\textnormal{FOC}} = 10^{-4}\) for the \ac{foc} condition and \(\tau_J = 10^{-12}\) for the objective function. The optimal solution for this problem is given by $\mu^* = \begin{bmatrix}
    1.4246656 & \pi
\end{bmatrix}^T$ with $J(\mu^*) = 2.39170787$. Until the end of this section, the optimal parameter $\mu^*$ will serve as the reference solution, against which the accuracy and efficiency of the \ac{phktr} algorithm will be measured. \\

The results of the \ac{phktr} algorithm when utilizing the quadratic Matérn kernel with different shape parameters are displayed in Table \ref{table::2DPDE_IMQ}. For the target function $J$ in this example, unlike the one in Section \ref{subsection_numerical_examples_1D}, we can not compute the \ac{rkhs}-norm exactly. We therefore estimate the \ac{rkhs}-norm using \eqref{eq:computation_rkhs}. To that end we compute a global interpolant for $J$ using $n$ randomly sampled parameters $\mu \in \mathcal{P}$. By \eqref{eq:estimate_rkhsnorm} this estimate converges towards $\Vert J \Vert_{\mathcal{H}_k(\mathcal{P})}$ for $n \rightarrow \infty$. Note that we have omitted the \ac{fom} evaluations required to estimate the \ac{rkhs}-norm in Table \ref{table::2DPDE_IMQ}. We made this decision for two primary reasons. Firstly, this task lends itself to easy parallelization. Secondly, we can obtain these FOM solutions by employing a coarser mesh in solving the primal equation (\ref{primal_eqn_pymor}). Consequently, the runtime for estimating the \ac{rkhs}-norm does not significantly impact the overall runtime of the algorithm. \\

\begin{table}[ht!]
\centering
\small
\renewcommand{\arraystretch}{1.2}
\begin{tabular}{|c||c|c|c|c|c|}
\hline
kernel shape parameter $\varepsilon$  & \ac{av} \ac{fom} evaluations  & \ac{av} \ac{foc} condition & \ac{av} relative error in $J$   \\
\hline 
\hline
$0.1$ & $7.6$ & $2 \cdot 10^{-5}$ & $2 \cdot 10^{-10}$ \\
\hline 
$0.2$ & $6.8$ & $5 \cdot 10^{-6}$ & $2 \cdot 10^{-11}$ \\
\hline
$0.3$ & $6.8$  & $2 \cdot 10^{-5}$ & $1 \cdot 10^{-10}$ \\
\hline
$0.4$ & $6.8$ & $5 \cdot 10^{-6}$ & $2 \cdot 10^{-11}$ \\
\hline 
$0.5$ & $7.2$ & $1 \cdot 10^{-5}$ & $5 \cdot 10^{-11}$ \\
\hline
$0.6$ & $8.8$ & $4 \cdot 10^{-5}$ & $3 \cdot 10^{-10}$ \\
\hline
\end{tabular}
\caption{Performance and accuracy of the \ac{phktr} algorithm using the quadratic Matérn kernel to solve the 2D-PDE constrained optimization problem for five optimization runs with randomly sampled initial parameters $\mu^{(0)} \in \mathcal{P}$.}
\label{table::2DPDE_IMQ}
\end{table} 

The comparison of the \ac{phktr} algorithm (using the kernel shape parameter $\varepsilon = 0.4$) with the \texttt{L-BFGS-B} and \texttt{trust-constr} algorithms from \texttt{scipy} is displayed in Table \ref{table_results_2D_compare}. 
\begin{table}[ht!]
\centering
\small
\renewcommand{\arraystretch}{1.2}
\begin{tabular}{|c||c|c|c|c|c|}
\hline
method & \ac{av} \ac{fom} evaluations & \ac{av} relative error in $J$  \\
\hline 
\hline
\ac{phktr} with $\varepsilon = 0.4$  &  $6.8$  & $2 \cdot 10^{-11}$ \\
\hline 
\texttt{L-BFGS-B} & $7.0$ & $3 \cdot 10^{-11}$ \\
\hline
\texttt{trust-constr} ($\tau_{\textnormal{FOC}} = 10^{-4}$) & $7.8$ & $1 \cdot 10^{-3}$ \\
\hline 
\texttt{trust-constr} ($\tau_{\textnormal{FOC}} = 10^{-12}$) & $15.6$ & $4 \cdot 10^{-8}$ \\
\hline
\end{tabular}
\caption{Comparison of the \ac{phktr} algorithm using the quadratic Matérn kernel to solve the 2D optimization problem for five optimization runs with randomly sampled initial parameters $\mu^{(0)} \in \mathcal{P}$ with the \texttt{L-BFGS-B} and \texttt{trust-constr} algorithm using tolerances of $\tau_{\textnormal{FOC}} = 10^{-4}$ and $\tau_{\textnormal{FOC}} = 10^{-12} $.}
\label{table_results_2D_compare}
\end{table}

The results indicate that the \texttt{trust-constr} method encounters difficulties identifying the optimal parameter \(\mu^*\). This issue arises because 
\(\mu_2^*\) lies on the boundary of the parameter space \(\mathcal{P}\). Specifically, \texttt{trust-constr} employs the Lagrange gradient as its termination criterion, rather than a projected gradient, necessitating a tolerance of order of $10^{-12}$ to achieve an \ac{av} relative error in $J$ of order $10^{-8}$. Significant speedups of the \ac{phktr} algorithm over the \texttt{L-BFGS-B} solver are not to be expected in this example, since the objective function, shown in Figure~\ref{figure:Objective_function2D}, appears approximately convex under visual inspection. Consequently, the quasi-Newton approach employed by \texttt{L-BFGS-B} is already well suited to this problem structure and performs very efficiently. We now turn to an example where the \ac{phktr} algorithm outperforms \texttt{L-BFGS-B} in terms of \ac{fom} evaluations, demonstrating its potential advantages in more challenging optimization landscapes.

\subsection{12D \ac{pde} constraint optimization problem}
\label{sect:num12D}
For the 12D problem, we consider a problem formulated in \citep[Section 5.3]{Keil_2021}, which deals with stationary heat distribution in a building. While \cite{Keil_2021} considers the problem in ten parameter dimensions, a preprint by the same authors extends it to twelve parameter dimensions (see Section 4.2 in arXiv:2012.11653). We present this problem in detail, following these two references. As the objective functional $\mathcal{J}: H \times \mathcal{P} \rightarrow \mathbb{R}$ (where $H \subset H^1(X)$ denotes a suitable function space which accounts for the Robin boundary data of the \ac{pde}-constraint \eqref{eqn:12D_pde_robin}), a weighted $L^2$-error on the domain of interest $D \subseteq X := [0,2] \times [0,1] \subset \mathbb{R}^2$ together with a regularization term is considered
\begin{align*}
\mathcal{J}(u(\cdot;\mu); \mu) = 50 \int_D (u(x;\mu) - u^d(x))^2 \, dx + \frac{1}{2} \sum_{m=1}^{12} \sigma_m (\mu_m - \mu_m^d)^2 + 1.
\end{align*}
Here \( u^d \) denotes the desired state, \( \mu^d \) the desired parameter and the weights $(\sigma_m)_{m=1}^{12}$ will be specified below. The constant term $1$ is added to fulfill Assumption \ref{assumption_list} b) and does not influence the location of the local minimum. As PDE-constraint, we consider as in Section \ref{subsection_pde_contraint_num_example} the parameterized stationary heat equation, this time with Robin boundary data:
\begin{align}
-\nabla \cdot (\lambda(x; \mu) \nabla u(x; \mu)) &= f(x; \mu)  &&\textnormal{in } X, \nonumber \\
c(x; \mu) (\lambda(x; \mu) \nabla u(x; \mu) \cdot n(x))&= (u_{\textnormal{out}}(x) - u(x; \mu)) &&\textnormal{on } \partial X, \label{eqn:12D_pde_robin}
\end{align}
with parametric diffusion coefficient $\lambda(\cdot; \mu) \in L^\infty(X)$, source term $f(\cdot; \mu) \in L^2(X)$, outside temperature $u_{\textnormal{out}} \in L^2(\partial X)$, Robin function $c(\cdot; \mu) \in L^\infty(\partial X)$ and the outer unit normal $n: \partial X \rightarrow \mathbb{R}^2$. Deriving the weak formulation analogously to Section \ref{subsection_pde_contraint_num_example} yields  \begin{align*}
a(u, v; \mu) &:= \int_X \lambda(x; \mu) \nabla v(x) \cdot \nabla u(x;\mu) \, dx + \int_{\partial X} \frac{1}{c(s; \mu)} v(s) u(s;\mu) \, ds, \\
l(v; \mu) &:= \int_X f(x; \mu) v(x) \, dx +  \int_{\partial X} \frac{1}{c(s; \mu)} u_{\textnormal{out}}(s) v(s) \, ds,
\end{align*} for $v \in H$. Motivated by the goal of maintaining a specified temperature within a single 
room \(D\) of a building floor \(X\), we account for the presence of windows, heaters, doors, and walls in the design, compare Figure \ref{figure::building_floor}. In this figure, numbers $j$ indicate different components inside the building floor, where $j.$ represents a window, $j|$ a wall and $\underline{j}$ a door. The $j$-th heater is located under window $j$. \\
\begin{figure}[h!]
    \centering
    \includegraphics[width=0.9\linewidth]{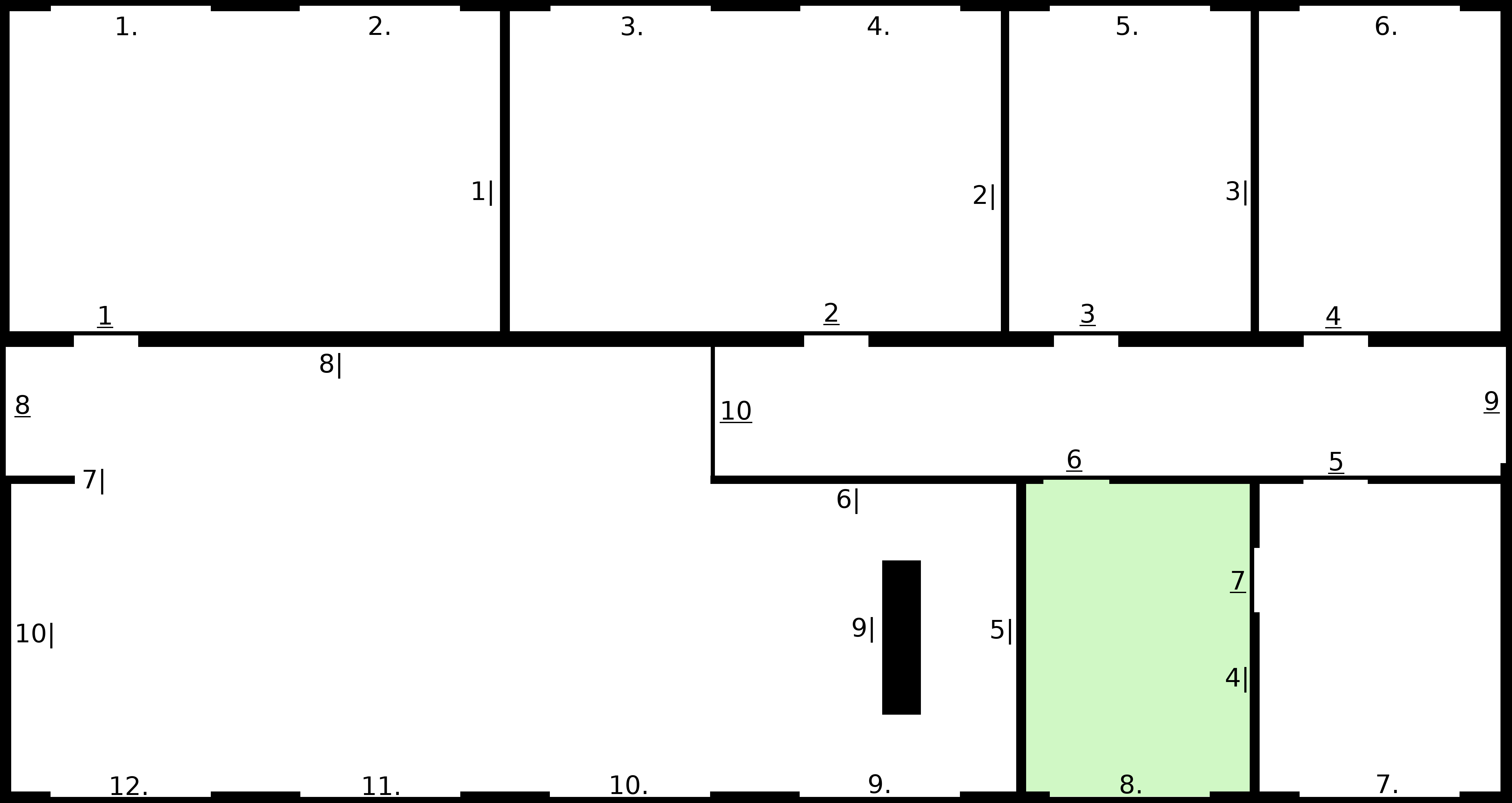}
    \caption{Figure 3 from \cite{Keil_2021}: The green room shows the domain of interest $D \subseteq X$.}
    \label{figure::building_floor}
\end{figure}

We seek to ensure a desired temperature $u^d(x) := 18\chi_D(x)$ and set \(\mu_m^d := 0 \; \forall \; m=1,\dots,12\). For the FOM discretization we use pyMOR's discretization toolkit. A cubic mesh is generated such that all spatial variations in the data functions extracted from Figure \ref{figure::building_floor} are fully resolved, yielding a discretised system with 
80601 degrees of freedom. We consider a 12D parameter set containing two door sets $\{ \underline{6} \}, \{ \underline{7}  \}$, seven heater sets \(\{1, 2\}\), \(\{3, 4\}\), \(\{5\}\), \(\{6\}\), \(\{7\}\), \(\{8\}\), and \(\{9, 10, 11, 12\}\), as well as three wall sets \(\{1|, 2|, 3|, 7|, 8| \}\), \(\{4|, 5|, 6| \}\), and \(\{9| \}\), where each set is governed by a single parameter component, resulting in 12 parameters. The set of admissible parameters is given by \(\mathcal{P} := [0.05, 0.2]^2 \times [0, 100]^7 \times [0.025, 0.1]^3\). We choose \begin{align*}
    (\sigma_m)_{1 \leq m \leq 12} = (\sigma_d, \sigma_d, 4\sigma_h, 4\sigma_h, \sigma_h, \sigma_h, \sigma_h, \sigma_h, 8\sigma_h, \sigma_w, \sigma_w, \sigma_w),
\end{align*} 
with \(\sigma_d = 1\), \(\sigma_h = 0.0005\) and $\sigma_w = 0.1$.  The other components of the data functions are fixed and thus not directly involved in the optimization process. They are chosen as follows: Air as well as the opened inside doors $\{ \underline{1}, \underline{2}, \underline{3}, \underline{4}, \underline{5}, \underline{10} \}$ have a diffusion coefficient of $0.5$, the outside doors $\{ \underline{8}, \underline{9} \}$ are closed with a constant diffusion coefficient of $0.001$. Further the outside wall $\{ 10|\}$ also has the diffusion coefficient $0.001$. All windows $\{ 1., \dots, 12. \}$ are supposed to be closed with diffusion constant $0.05$. The Robin data $c(\cdot; \mu)$ contains information about the outside wall $\{10| \}$, outside doors $\{ \underline{8}, \underline{9} \}$ and all windows. All other diffusion terms enter into $\lambda(\cdot; \mu)$. The source term contains all the information about the 12 heaters. The outside temperature is set to  $u_{\textnormal{out}} \equiv 5$. \\

As convergence criteria, we employ thresholds of $\tau_{\textnormal{FOC}} = 5 \cdot 10^{−4}$ for the \ac{foc} condition and $\tau_J = 10^{−12}$
for the objective function. Additionally, we restrict the algorithms to a maximum of $100$ iterations. The optimal value of the target function is  $J(\mu^*) = 5.813965$ (we refer to the \texttt{GitHub} repository for the optimal parameter $\mu^*$). In the \ac{phktr} algorithm (Algorithm \ref{kernel_TR_algorithm}) the Wendland kernel of second order provided good results. Table \ref{table:12DPDE_quadmatern} shows the performance and accuracy using different kernel shape parameters $\varepsilon$. Following the approach in \Cref{subsection_pde_contraint_num_example}, we estimated the \ac{rkhs}-norm using the method described in \Cref{sect:Hermite} and did not include the required \ac{fom} evaluations in Table \ref{table:12DPDE_quadmatern}. 
\begin{table}[ht!]
\centering
\small
\renewcommand{\arraystretch}{1.2}
\begin{tabular}{|c||c|c|c|c|c|}
\hline
kernel shape parameter $\varepsilon$  & \ac{av} \ac{fom} evaluations & \ac{av} \ac{foc} condition & \ac{av} error in $J$   \\
\hline 
\hline
$0.0006$ & $45.2$  & $ 6.8 \cdot 10^{-4}$ & $ 6.6\cdot 10^{-5}$ \\
\hline
$0.0008$ & $43.4$ & $4.6 \cdot 10^{-4}$ & $4.9 \cdot 10^{-5}$ \\ 
\hline 
$0.001$ & $52.8$ & $5.1 \cdot 10^{-4}$ & $4.9 \cdot 10^{-5}$ \\
\hline
\end{tabular}
\caption{Performance and accuracy of the \ac{phktr} algorithm using the Wendland kernel of second order to solve the 12D-PDE constrained optimization problem for five optimization runs with randomly sampled initial parameters $\mu^{(0)} \in \mathcal{P}$.}
\label{table:12DPDE_quadmatern}
\end{table} 

The comparison between the \ac{phktr} algorithm with the \texttt{L-BFGS-B} and the \texttt{trust-constr} algorithm is shown in Table \ref{table_results_12D_compare}.
\begin{table}[ht!]
\centering
\small
\renewcommand{\arraystretch}{1.2}
\begin{tabular}{|c||c|c|c|c|c|}
\hline
method & \ac{av} \ac{fom} evaluations & \ac{av} relative error in $J$  \\
\hline 
\hline
\ac{phktr} with $\varepsilon = 0.0008 $  & $43.4$  & $4.9 \cdot 10^{-5}$  \\
\hline 
\texttt{L-BFGS-B} & $54.2$  & $1.3 \cdot 10^{-5}$  \\
\hline
\texttt{trust-constr} & $75.0$ & $7.8 \cdot 10^{-4}$ \\
\hline 
\end{tabular}
\caption{Comparison of the \ac{phktr} algorithm using the Wendland kernel of second order to solve the 12D optimization problem for five optimization runs with randomly sampled initial parameters $\mu^{(0)} \in \mathcal{P}$ with the \texttt{L-BFGS-B} and \texttt{trust-constr} algorithm.}
\label{table_results_12D_compare}
\end{table}
Using a kernel shape parameter of $\varepsilon = 0.0008$, the \ac{phktr} algorithm outperforms both \texttt{scipy} algorithms in terms of \ac{fom} evaluations by $20\%$ and $41\%$, respectively. We remark that the \texttt{trust-constr} method once terminated due to the maximum amount of iterations and not due to the \ac{foc} condition. We observe that for this reason, the \texttt{trust-constr} solver performs one order of magnitude worse in terms of relative error in the objective function $J$. This contributes, as in the 2D example stated in \Cref{subsection_pde_contraint_num_example}, to the fact that the optimal $\mu^*$ has components on the boundary of $\mathcal{P}$, specifically $\mu^*_1, \mu^*_2, \mu^*_{10}, \mu^*_{11}$ and $\mu^*_{12}$, causing difficulties for the \texttt{trust-constr} 
solver. We note that a substantial outperformance over the \texttt{L-BFGS-B} or \texttt{trust-constr} method can not be expected, since all approaches rely exclusively on a history of sampled data along the optimization trajectory. Ultimately, the quality and informativeness of the available data become saturated, limiting the potential for further improvement in surrogate accuracy and, consequently, in optimization performance. Compared to the proposed \ac{phktr} algorithm, the approach introduced in \cite{Keil_2021}, which use reduced basis techniques to build the surrogate model, achieve better results in terms of FOM evaluations. This outcome is expected, since a reduced-basis-based model inherently encodes the physics of the underlying PDE, whereas our framework is purely data-driven. We refer to \Cref{sect:Conclusion}, where we outline why our approach is more flexible. 

\section{Conclusion and outlook}\label{sect:Conclusion}
In this work, we introduced a novel approach to construct surrogate models in the context of \ac{tr}-based optimization. In \Cref{section_hermite_TR_kernel}, the main section of this study, we gave a comprehensive discussion of the proposed \ac{phktr} algorithm, including a convergence proof under reasonable assumptions. One main feature of the proposed Algorithm is the definition of the \ac{tr} based on the upper bound of the kernel interpolation error - a difference to most \ac{tr} methods in literature, which restrict the \ac{tr} to balls. In \Cref{chapter_numerical_examples} we demonstrated the effectiveness of the algorithm on three different optimization problems and were able to perform better than the \texttt{scipy} implementation of the \texttt{L-BFGS-B} algorithm. \\

We outlined the strengths and weaknesses of the \ac{hktr} method. Numerical experiments detailed in \Cref{sect:num12D} indicate that the \ac{hktr} algorithm is outperformed by reduced basis surrogate models, where the (linear) FOM can be efficiently reduced and subsequently the reduced model serves as a surrogate. In this context, combining the \ac{hktr} algorithm with reduced basis methods - similar to \cite{haasdonk2023new} - could harness the strengths of each method. Specifically, the \ac{hktr} is applied to a reduced model that is adaptively updated with FOM data whenever an a posteriori error estimate reveals that the surrogate has become insufficiently accurate.
Nevertheless, the pure \ac{hktr} method exhibits considerably greater flexibility. It can be applied to nonlinear PDE-constrained problems, where constructing an appropriate reduced-basis surrogate requires more advanced techniques than in the setting of a linear coercive \ac{pde}. Furthermore, the \ac{hktr} algorithm can also be utilized for high-dimensional optimization tasks unrelated to PDEs. As long as the target function lives in the \ac{rkhs} associated with the chosen kernel, the approach will deliver favorable results. The ability to apply the \ac{hktr} algorithm to a wide range of optimization problems is undoubtedly a significant strength. \\ 

As discussed in \Cref{chapter_numerical_examples}, the kernel shape parameter $\varepsilon$ significantly influences the performance of the proposed algorithm. To reduce or eliminate this dependency, one possible direction is to incorporate an adaptive shape parameter that is updated at each iteration. In this context, we briefly explored two conceptual approaches, which, however, were not pursued or developed in detail. In the first, the shape parameter is adjusted for the entire surrogate model, influencing it globally rather than only modifying the region around the current iterate. Another approach would be to assign distinct shape parameters for each newly selected iterate $\mu^{(i)}$, while preserving those used for previous iterates. This would produce a surrogate model that is accurate not only locally but also potentially along the entire optimization path. However, such an approach would yield kernel matrices that are no longer symmetric. To the best of our knowledge, this aspect has not been thoroughly investigated from a theoretical standpoint, and fundamental questions, such as the solvability of the resulting linear systems, would naturally arise.



\section*{Acknowledgments}
The authors acknowledge the funding of the project by
the Deutsche Forschungsgemeinschaft (DFG, German Research Foundation) under
number 540080351 and Germany’s Excellence Strategy - EXC 2075 - 390740016.

\addcontentsline{toc}{section}{References}
\bibliography{sn-bibliography}

\appendix
\section{Proof of Theorem \ref{auxilary_statements_from_YM}}
\label{appendix_proof_YM}

\begin{proof}
We start by proving the following auxiliary result: (\ref{subproblem_armijo_cond}) and (\ref{subproblem_constraint_cond}) (for $\mu$ instead of $\mu^{(i)}(j)$) are satisfied for all $\mu$ of the form (\ref{eqn:lsForConvergence}) that satisfy
\begin{align}
\label{eqn:theorem_decrease_agc_first_aux_result}
\left\Vert\mu^{(i)} - \mu \right\Vert \leq \min \left\{ \frac{(1-\kappa_{\textnormal{arm}})\cos \Phi^{(i)} \left\Vert\nabla
\hat{J}^{(i)}(\mu^{(i)})\right\Vert}{C_{\nabla \hat{J}}^{(i)}}, \frac{\left(c^{(i)}(\mu^{(i)})\right)^2}{\left(C_c^{(i)} \right)^2} \right\}.
\end{align}
If $\left\Vert \nabla \hat{J}^{(i)} (\mu^{(i)}) \right\Vert = 0$, then (\ref{eqn:theorem_decrease_agc_first_aux_result}) implies $\Vert\mu^{(i)} - \mu\Vert = 0$, thus $ \mu = \mu^{(i)}$. Therefore, (\ref{subproblem_armijo_cond}) and (\ref{subproblem_constraint_cond}) hold trivially (for $\mu$ instead of $\mu^{(i)})$. Now we consider the case $\left\Vert \nabla \hat{J}^{(i)} (\mu^{(i)})\right\Vert > 0$. We introduce the abbreviation \begin{align}
\label{eqn_def_descent_dir}
\nabla_{p^{(i)}} \hat{J}^{(i)} (\mu) := \left(\nabla \hat{J}^{(i)}(\mu) \right)^T p^{(i)}.
\end{align}
For a descent direction $p^{(i)}$ we have \begin{align}
\label{eqn_descent_dir}
\nabla_{p^{(i)}} \hat{J}^{(i)} (\mu^{(i)}) = - \left\Vert \nabla \hat{J}^{(i)} (\mu^{(i)})\right\Vert \left\Vert p^{(i)}\right\Vert \cos \Phi^{(i)} < 0.
\end{align}
Let us consider the equation \begin{align}
\label{theorem3.2_eqn316}
\nabla_{p^{(i)}} \hat{J}^{(i)} (\mu) = -  \kappa_{\textnormal{arm}} \left\Vert \nabla \hat{J}^{(i)} (\mu^{(i)}) \right\Vert  \left\Vert p^{(i)} \right\Vert \cos \Phi^{(i)},
\end{align}
which has at least one solution $\tilde{\mu}$ of the form (\ref{eqn:lsForConvergence}). We prove this by contradiction. Assume \eqref{theorem3.2_eqn316} has no solution. As $\nabla \hat{J}^{(i)}$ is Lipschitz continuous according to the assumption of this theorem, it is also continuous.
\begin{enumerate}[label=\roman*)]
\item Assume that $\nabla_{p^{(i)}} \hat{J}^{(i)} (\mu) < -  \kappa_{\textnormal{arm}} \left\Vert \nabla \hat{J}^{(i)} (\mu^{(i)}) \right\Vert  \left\Vert p^{(i)} \right\Vert \cos \Phi^{(i)}$ holds for all $\mu$ of the form (\ref{eqn:lsForConvergence}). Using Lagrange's mean value theorem yields existence of a $\bar{\mu} \in \left\{ \lambda \mu + (1 - \lambda) \mu^{(i)} \; | \; \lambda \in (0,1) \right\}$, s.t. \begin{align}
\label{eqn_mu_bar_kleiner}
\hat{J}^{(i)}(\mu) - \hat{J}^{(i)}(\mu^{(i)}) &= \left( \nabla \hat{J}^{(i)}(\bar{\mu}) \right)^T(\mu - \mu^{(i)}).
\end{align}
Note that $\bar{\mu}$ is also of the form (\ref{eqn:lsForConvergence}), as \begin{align*}
\bar{\mu} = \lambda \mu + (1 - \lambda) \mu^{(i)} = \lambda ( \mu^{(i)} + \alpha p^{(i)} ) + (1 - \lambda) \mu^{(i)} = \mu^{(i)} + \lambda \alpha p^{(i)},
\end{align*} 
using a scaled step length $\bar{\alpha} := \lambda \alpha \geq 0$. Therefore, \begin{align}
\label{mu_bar_follows_assumption}
\nabla_{p^{(i)}} \hat{J}^{(i)} (\bar{\mu}) < -  \kappa_{\textnormal{arm}} \left\Vert \nabla \hat{J}^{(i)} (\mu^{(i)}) \right\Vert \left\Vert p^{(i)} \right\Vert \cos \Phi^{(i)},
\end{align}  holds and we can use this inequality to conclude \begin{align*}
\hat{J}^{(i)}(\mu) - \hat{J}^{(i)}(\mu^{(i)}) &= \left( \nabla \hat{J}^{(i)}(\bar{\mu}) \right)^T(\mu - \mu^{(i)}) \\
&= \left( \nabla \hat{J}^{(i)}(\bar{\mu}) \right)^T (\alpha p^{(i)}) \\
\overset{(\ref{eqn_def_descent_dir})}&{=} \alpha \; \nabla_{p^{(i)}} \hat{J}^{(i)} (\bar{\mu}) \\
\overset{(\ref{mu_bar_follows_assumption})}&{<}  - \alpha \kappa_{\textnormal{arm}} \left\Vert \nabla \hat{J}^{(i)} (\mu^{(i)}) \right\Vert \left\Vert p^{(i)} \right\Vert \cos \Phi^{(i)} \\
&= - \kappa_{\textnormal{arm}} \left\Vert \nabla \hat{J}^{(i)} (\mu^{(i,l)}) \right\Vert  \left\Vert\mu - \mu^{(i)} \right\Vert \cos \Phi^{(i)}.
\end{align*}
for all $\mu$ of the form (\ref{eqn:lsForConvergence}), even for the case $\Vert\mu - \mu^{(i)}\Vert \rightarrow \infty $. This indicates $\hat{J}^{(i)}(\mu) \rightarrow - \infty$, which contradicts Assumption \ref{assumption_list} b) , namely that $\hat{J}^{(i)}$ is bounded from below. Hence, this case is not possible. 
\item Now assume $\nabla_{p^{(i)}} \hat{J}^{(i)} (\mu) > -  \kappa_{\textnormal{arm}} \left\Vert \nabla \hat{J}^{(i)} (\mu^{(i)}) \right\Vert \left\Vert p^{(i)} \right\Vert \cos \Phi^{(i)}$ holds for all $\mu$ of the form (\ref{eqn:lsForConvergence}). Analogously, we can conclude \begin{align*}
\hat{J}^{(i)}(\mu) - \hat{J}^{(i)}(\mu^{(i)}) > - \kappa_{\textnormal{arm}} \left\Vert \nabla \hat{J}^{(i)} (\mu^{(i)}) \right\Vert \left\Vert\mu - \mu^{(i)} \right\Vert \cos \Phi^{(i)}
\end{align*}
for all $\mu$ of the form (\ref{eqn:lsForConvergence}). By choosing $\alpha=0$ we obtain $\mu = \mu^{(i)}$ and therefore \begin{align*}
0 = \hat{J}^{(i)}(\mu) - \hat{J}^{(i)}(\mu^{(i)}) > - \kappa_{\textnormal{arm}} \left\Vert \nabla \hat{J}^{(i)} (\mu^{(i)}) \right\Vert \underbrace{\left\Vert\mu - \mu^{(i)} \right\Vert}_{ = 0} \cos \Phi^{(i)} = 0,
\end{align*}
which is a contradiction. Consequently, it is also impossible for this case to occur. 
\end{enumerate}
Because neither case i) nor case ii) holds, we can conclude using the intermediate value theorem that a solution $\tilde{\mu}$ of the form (\ref{eqn:lsForConvergence}) for (\ref{theorem3.2_eqn316}) has to exist. 
As $\nabla \hat{J}^{(i)}$ is Lipschitz continuous by assumption, this solution satisfies \begin{align*}
\left\Vert \mu^{(i)} -\tilde{\mu} \right\Vert &\geq \frac{ \left\Vert\nabla \hat{J}^{(i)}(\tilde{\mu}) - \nabla \hat{J}^{(i)}(\mu^{(i)}) \right\Vert}{C_{\nabla \hat{J}}^{(i)}} = \frac{ \left\Vert\nabla \hat{J}^{(i)}(\tilde{\mu}) - \nabla \hat{J}^{(i)}(\mu^{(i)}) \right\Vert \left\Vert p^{(i)} \right\Vert}{C_{\nabla \hat{J}}^{(i)} \left\Vert p^{(i)} \right\Vert} \\
&\geq \frac{ \left|\nabla_{p^{(i)}} \hat{J}^{(i)} (\tilde{\mu}) - \nabla_{p^{(i)}} \hat{J}^{(i)} (\mu^{(i)}) \right|}{C_{\nabla \hat{J}}^{(i)} \left\Vert p^{(i)} \right\Vert} = \frac{(1 - \kappa_{\textnormal{arm}})\cos \Phi^{(i)} \left\Vert \nabla \hat{J}^{(i)}(\mu^{(i)}) \right\Vert}{C_{\nabla \hat{J}}^{(i)}},
\end{align*}
where we used the Cauchy-Schwarz inequality to obtain the second inequality and (\ref{eqn_descent_dir}) as well as (\ref{theorem3.2_eqn316}) to obtain the last equality. We show that (\ref{subproblem_armijo_cond}) holds for all $\mu$ of the form (\ref{eqn:lsForConvergence}) satisfying \begin{align}
\label{eqn_mu_mui}
\left\Vert\mu^{(i)} - \mu \right\Vert \leq \frac{(1 - \kappa_{\textnormal{arm}})\cos \Phi^{(i)} \left\Vert \nabla \hat{J}^{(i)}(\mu^{(i)}) \right\Vert}{C_{\nabla \hat{J}}^{(i)}}.
\end{align}
First note that for $\bar{\mu}$ introduced in \eqref{eqn_mu_bar_kleiner} the following holds \begin{align}
\label{eqn_mu_bar_kleiner_2}
\left\Vert\bar{\mu} - \mu^{(i)} \right\Vert = \left\Vert \lambda \mu + (1 - \lambda) \mu^{(i)} - \mu^{(i)} \right\Vert = \left\Vert \lambda ( \mu - \mu^{(i)} ) \right\Vert \leq \left\Vert \mu - \mu^{(i)} \right\Vert.
\end{align}
Let $\mu$ be of form \eqref{eqn:lsForConvergence} s.t. \eqref{eqn_mu_mui} holds. By first utilizing the Cauchy-Schwarz inequality, followed by the Lipschitz continuity of $\nabla \hat{J}^{(i)}$ and inequality (\ref{eqn_mu_bar_kleiner_2}), we obtain
\begin{align*}
\left|\nabla_{p^{(i)}} \hat{J}^{(i)} (\bar{\mu}) - \nabla_{p^{(i)}} \hat{J}^{(i)} (\mu^{(i)}) \right| &\leq \left\Vert\nabla \hat{J}^{(i)}(\bar{\mu}) - \nabla \hat{J}^{(i)}(\mu^{(i)}) \right\Vert \left\Vert p^{(i)} \right\Vert \\
 &\leq C_{\nabla \hat{J}}^{(i)} \; \left\Vert\bar{\mu} - \mu^{(i)} \right\Vert \left\Vert p^{(i)}\right\Vert\\
 &\leq C_{\nabla \hat{J}}^{(i)} \; \left\Vert\mu - \mu^{(i)}\right\Vert \left\Vert p^{(i)}\right\Vert \\
\overset{(\ref{eqn_mu_mui})}&{\leq} C_{\nabla \hat{J}}^{(i)} \; \frac{(1 - \kappa_{\textnormal{arm}})\cos \Phi^{(i)} \left\Vert \nabla \hat{J}^{(i)}(\mu^{(i)}) \right\Vert}{C_{\nabla \hat{J}}^{(i)}} \left\Vert p^{(i)} \right\Vert \\ 
&= (1 - \kappa_{\textnormal{arm}})\cos \Phi^{(i)} \left\Vert \nabla \hat{J}^{(i)}(\mu^{(i)}) \right\Vert \left\Vert p^{(i)} \right\Vert \\
\overset{(\ref{eqn_descent_dir})}&{=} - \kappa_{\textnormal{arm}} \cos \Phi^{(i)} \left\Vert \nabla \hat{J}^{(i)}(\mu^{(i)}) \right\Vert \left\Vert p^{(i)} \right\Vert - \nabla_{p^{(i)}} \hat{J}^{(i)} (\mu^{(i)}),
\end{align*}
yielding \begin{align}
\label{eqn_upper_bound_nabla_derivative}
\nabla_{p^{(i)}} \hat{J}^{(i)} (\bar{\mu}) \leq - \kappa_{\textnormal{arm}} \cos \Phi^{(i)} \left\Vert \nabla \hat{J}^{(i)}(\mu^{(i)}) \right\Vert \left\Vert p^{(i)} \right\Vert.
\end{align}
We can now conclude that
\begin{align*}
\hat{J}^{(i)}(\mu) - \hat{J}^{(i)}(\mu^{(i)}) &= \left( \nabla \hat{J}^{(i)}(\bar{\mu}) \right)^T(\mu - \mu^{(i)} ) \\
\overset{(\ref{eqn_def_descent_dir})}&{=} \alpha \; \nabla_{p^{(i)}} \hat{J}^{(i)} (\bar{\mu}) \\
\overset{(\ref{eqn_upper_bound_nabla_derivative})}&{\leq} - \alpha \kappa_{\textnormal{arm}} \left\Vert \nabla \hat{J}^{(i)}(\mu^{(i)}) \right\Vert \left\Vert p^{(i)} \right\Vert  \cos \Phi^{(i)} \\
&=  - \kappa_{\textnormal{arm}} \left\Vert \nabla \hat{J}^{(i)}(\mu^{(i)}) \right\Vert\ \left\Vert\mu - \mu^{(i)} \right\Vert \cos \Phi^{(i)}, 
\end{align*}
thus (\ref{subproblem_armijo_cond}) holds. Due to the Hölder continuity of $c^{(i)}$ with $\alpha_{\textnormal{Höl}} = \sfrac{1}{2}$, a solution $\tilde{\tilde{\mu}}$ of $c^{(i)}(\mu) = 0$ satisfies \begin{alignat*}{2}
&\left\Vert \mu^{(i)} - \tilde{\tilde{\mu}} \right\Vert^{\frac{1}{2}} &&\geq \frac{ \left|c^{(i)}(\mu^{(i)}) - c^{(i)}(\tilde{\tilde{\mu}})\right|}{C_c^{(i)}} = \frac{c^{(i)}(\mu^{(i)})}{C_c^{(i)}} \\
\Longleftrightarrow \quad &\left\Vert \mu^{(i)} - \tilde{\tilde{\mu}} \right\Vert &&\geq \frac{\left(c^{(i)}(\mu^{(i)})\right)^2}{\left(C_c^{(i)} \right)^2}
\end{alignat*}
which means that (\ref{subproblem_constraint_cond}) holds for all $\mu$ of the form (\ref{eqn:lsForConvergence}) satisfying \begin{align*}
\left\Vert \mu^{(i)} - \mu \right\Vert \leq \frac{\left(c^{(i)}(\mu^{(i)})\right)^2}{\left(C_c^{(i)} \right)^2},
\end{align*} 
because then we obtain \begin{align*}
\left|c^{(i)}(\mu^{(i)}) - c^{(i)}(\mu) \right| &\leq C_c^{(i)} \left\Vert \mu^{(i)} - \mu \right\Vert^{\frac{1}{2}} \leq C_c^{(i)} \frac{c^{(i)}(\mu^{(i)})}{C_c^{(i)}} = c^{(i)}(\mu^{(i)}),
\end{align*}
yielding the desired result $c^{(i)} (\mu) \geq 0$, as we assume $c^{(i)}(\mu^{(i)}) > 0$. This proves the auxiliary result (\ref{eqn:theorem_decrease_agc_first_aux_result}). \\

We continue by proving another auxiliary statement, namely: The \ac{agc} point $\mu_{\textnormal{AGC}}^{(i)}$ satisfies \begin{align}
\label{eqn:theorem_descrease_agc_second_aux}
\left\Vert \mu_{\textnormal{AGC}}^{(i)} - \mu^{(i)} \right\Vert \geq \min \left\{ \kappa_{\nabla \hat{J}}^{(i)} \left\Vert\nabla \hat{J}^{(i)} (\mu^{(i)}) \right\Vert, \kappa_{\textnormal{bt}} \frac{c^{(i)}(\mu^{(i)})^2}{(C_c^{(i)})^2} \right\},
\end{align}
where $\kappa_{\nabla \hat{J}}^{(i)} := \min \left\{ 1, \frac{\kappa_{\textnormal{bt}}(1 - \kappa_{\textnormal{arm}})\cos \Phi^{(i)}}{C_{\nabla \hat{J}}^{(i)}} \right\}$. The first line search point  $\mu^{(i)}(0) = \mu^{(i)} + p^{(i)}$ satisfies \begin{align}
\label{eqn_upper_bound_first_iterate_agc_proof}
\left\Vert \mu^{(i)}(0) - \mu^{(i)} \right\Vert = \left\Vert p^{(i)} \right\Vert = \left\Vert \nabla \hat{J}^{(i)} (\mu^{(i)}) \right\Vert.
\end{align}
If $\mu^{(i)}(0)$ satisfies (\ref{eqn:theorem_decrease_agc_first_aux_result}), it gets accepted as $\mu_{\textnormal{AGC}}^{(i)}$. In this case, the required upper bound \eqref{eqn:theorem_descrease_agc_second_aux} holds trivially due to (\ref{eqn_upper_bound_first_iterate_agc_proof}) as $ \kappa_{\nabla \hat{J}}^{(i)} \leq 1$: \begin{align*}
\left\Vert \mu^{(i)}(0) - \mu^{(i)} \right\Vert \geq \min \left\{ \kappa_{\nabla \hat{J}}^{(i)} \; \left\Vert \nabla \hat{J}^{(i)} (\mu^{(i)}) \right\Vert, \kappa_{\textnormal{bt}} \frac{c^{(i)}(\mu^{(i)})^2}{(C_c^{(i)})^2} \right\}.
\end{align*} 
 If $\mu^{(i)}(0)$ is not accepted as $\mu_{\textnormal{AGC}}^{(i)}$, backtracking occurs. However, it must stop before \begin{align}
\label{gleichung319}
\left\Vert\mu^{(i)}(j) - \mu^{(i)} \right\Vert \leq \kappa_{\textnormal{bt}} \min \left\{ \frac{(1 - \kappa_{\textnormal{arm}})\cos\Phi^{(i)} \left\Vert \nabla \hat{J}^{(i)}(\mu^{(i)}) \right\Vert}{C_{\nabla \hat{J}}^{(i)}} , \frac{\left(c^{(i)}(\mu^{(i)})\right)^2}{\left(C_c^{(i)}\right)^2} \right\}
\end{align}
holds, since otherwise the previous backtracking point $\mu^{(i)}(j-1)$ would satisfy  (\ref{eqn:theorem_decrease_agc_first_aux_result}) and thus be already accepted as $\mu_{\textnormal{AGC}}^{(i)}$. Therefore, (\ref{gleichung319}) cannot hold for $\mu_{\textnormal{AGC}}^{(i)}$. According to these two arguments, $\mu^{(i)}_{\textnormal{AGC}}$ satisfies (\ref{eqn:theorem_descrease_agc_second_aux}). \\

Now the desired result \eqref{theorem3.2_eqn315} follows immediately from (\ref{subproblem_armijo_cond}) and (\ref{eqn:theorem_descrease_agc_second_aux}).
\end{proof}

\section{Computation of the \ac{rkhs}-norm from Section \ref{subsection_numerical_examples_1D}}
\label{appendix_rhhsnorm}
Following \cite[Theorem 10.12]{wendland_2004}, the \ac{rkhs}-norm - corresponding to a translation invariant \ac{spd} kernel $k$ with $\phi \in C(\mathbb{R}) \cap L^1(\mathbb{R})$ - of a univariate function $J \in L^2(\mathbb{R}) \cap C(\mathbb{R})$,  s.t $\frac{\mathcal{F}(J)}{\sqrt{\mathcal{F}(\phi)}} \in L^2(\mathbb{R})$,  can be computed via 
\begin{align*}
    \Vert J\Vert_{\mathcal{H}_{k}}^2 = \frac{1}{\sqrt{2 \pi}} \int_{\mathbb{R}} \frac{|\mathcal{F}(J)(\omega)|^2}{\mathcal{F}(\phi)(\omega)} d\omega,
\end{align*}
where $\mathcal{F}$ denotes the Fourier transformation, given for $J \in L^1(\mathbb{R})$ by  \begin{align*}
    \mathcal{F}(J)(\omega) := \frac{1}{\sqrt{2 \pi}} \int_{\mathbb{R}} J(x) \exp(- i \omega x) dx.
\end{align*}
For the Fourier transformations of the target function $J(\mu) = - \exp\left(-\mu^2\right) + 3 \exp(- 0.001 \mu^2)$ and the radial basis function of the Gaussian kernel (in one dimension) $\phi_G(r;\varepsilon) := \exp(- \varepsilon^2 r^2) $ we obtain: \begin{align*}
    \mathcal{F}(J)(\omega) &= - \exp\left(- \frac{\omega^2}{4}\right) + 67082 \exp(-250 \omega^2) \\
    \mathcal{F}(\phi_G)(\omega) &= \frac{1}{\sqrt{2 \varepsilon^2}} \exp\left(- \frac{\omega^2}{4 \varepsilon^2}\right).
\end{align*}

Together this yields for $\varepsilon^2 > \sfrac{1}{2}$ the following result {\small \[
    \Vert J \Vert_{\mathcal{H}_{k}(\mathcal{P})}^2 = \frac{\varepsilon^2 \left(\sqrt{2000\varepsilon^{2} - 1} \left( \sqrt{1001 \varepsilon ^{2} - 1} - 33541 \cdot 2^{\frac{5}{2}} \sqrt{2\varepsilon^{2} - 1}\right) + 8999989448 \sqrt{2 \varepsilon^{2} - 1} \sqrt{1001\varepsilon^{2} - 1}\right)}{\sqrt{2\varepsilon^{2} - 1} \sqrt{1001\varepsilon^{2} - 1} \sqrt{2000\varepsilon^{2} - 1}}.
\]}




\end{document}